\documentclass[reqno, 11pt]{amsart}
\usepackage{amsmath,amsfonts, comment}
\usepackage{xcolor}
\usepackage{color}
\usepackage{hyperref}
\usepackage[inner=1.0in,outer=1.0in,bottom=1.0in, top=1.0in]{geometry}

\newcommand{\sm}{\left(\begin{smallmatrix}}
\newcommand{\esm}{\end{smallmatrix}\right)}
\newcommand{\bpm}{\begin{pmatrix}}
\newcommand{\ebpm}{\end{pmatrix}}

\newcommand{\mfp}{\mathfrak{p}}
\newcommand{\mfo}{\mathfrak{o}}

\newcommand{\A}{\mathbb{A}}
\newcommand{\R}{\mathbb{R}}
\newcommand{\C}{\mathbb{C}}
\renewcommand{\H}{\mathbb{H}}
\newcommand{\HH}{\mathbb{H}}
\newcommand{\Q}{\mathbb{Q}}
\newcommand{\Z}{\mathbb{Z}}
\newcommand{\X}{\mathbb{X}}

\newcommand{\scrA}{\mathcal{A}}
\newcommand{\scrP}{\mathcal{P}}
\newcommand{\scrN}{\mathcal{N}}

\newcommand{\Siegel}{\mathfrak{S}}

\newcommand{\ba}{\mathfrak{a}}
\newcommand{\bd}{\mathfrak{d}}
\newcommand{\bp}{\mathfrak{p}}
\newcommand{\bo}{\mathfrak{o}}
\newcommand{\bn}{\mathfrak{n}}

\DeclareMathOperator{\GL}{GL}
\DeclareMathOperator{\PGL}{PGL}

\DeclareMathOperator{\PO}{PO}

\DeclareMathOperator{\PSO}{PSO}
\DeclareMathOperator{\PU}{PU}
\DeclareMathOperator{\K}{K}

\DeclareMathOperator{\bM}{M}

\DeclareMathOperator{\tr}{tr}
\DeclareMathOperator{\sgn}{sgn}
\DeclareMathOperator{\Res}{Res}

\DeclareMathOperator{\sym}{Sym}

\DeclareMathOperator{\ord}{ord}

\DeclareMathOperator{\vol}{vol}

\DeclareMathOperator{\G}{G}

\DeclareMathOperator{\BorelB}{B}
\DeclareMathOperator{\N}{N}
\DeclareMathOperator{\maxK}{K}

\DeclareMathOperator{\CentZ}{Z}

\DeclareMathOperator{\finite}{finite}

\def\vecm{{\text{\boldmath$m$}}}

\definecolor{blue}{rgb}{0,0,1}
\definecolor{red}{rgb}{1,0,0}
\definecolor{green}{rgb}{0,.6,.2}
\definecolor{purple}{rgb}{1,0,1}

\long\def\red#1\endred{\textcolor{red}{#1}}
\long\def\blue#1\endblue{\textcolor{blue}{#1}}
\long\def\purple#1\endpurple{\textcolor{purple}{ #1}}
\long\def\green#1\endgreen{\textcolor{green}{#1}}

\newtheorem{theorem}{Theorem}[section]

\newtheorem{conjecture}[theorem]{Conjecture}
\newtheorem{proposition}[theorem]{Proposition}

\newtheorem{lemma}[theorem]{Lemma}

\theoremstyle{remark}
\newtheorem{remark}[theorem]{Remark}

\numberwithin{equation}{section}

\title{Linnik problem for Maass--Hecke cuspforms and effective multiplicity one theorem}
\author{Junehyuk Jung and Min Lee}

\address{Department of Mathematics, Brown University, Providence, RI 02912 USA}

\email{junehyuk{\textunderscore}jung@brown.edu}

\address{Fry Building, University of Bristol, Woodland road, BS8 1UG, United Kingdom}

\email{min.lee@bristol.ac.uk}

\thanks{We thank M. Young for sharing the idea of improving the exponent in Theorem \ref{thm:linnik},
which previously was $8/\alpha$.
J.J. thanks L. Silberman, Z. Shem-Tov, P. Humphries, P. Sarnak,  A. Zaman, and J. Thorner for many helpful discussions.
M.L. thanks A. Booker and M. Krishnamurthy for helpful discussions.
J.J. is partially supported by Sloan Research Fellowship and by NSF grant DMS-2050123. 
J.J. thanks school of mathematics of University of Bristol for hospitality and support during visit for collaboration with M. L.
M.L. is partially supported by Royal Society University Research Fellowship. }

\begin{document}

\begin{abstract}
We investigate two related problems concerning the dimension of joint eigenspaces of the Laplace--Beltrami operator and a finite set of Hecke operators on $\mathbb{X}=\PGL_2(\mathbb{Z})\backslash \mathbb{H}$. First, we consider Linnik problem for Maass--Hecke cuspforms. We prove that the dimension of such a joint eigenspace, for Maass--Hecke cuspforms with eigenparameter in $[T, T+1]$, associated to Hecke operators $T_p$ with $p < (\log T)^\alpha$ is $O_\epsilon (T^{{\frac{4}{\alpha}} + \epsilon})$. For this, we prove a new form of spectral large sieve inequality for symmetric-squares of Maass--Hecke cuspforms, by exploiting the fact that the forms under consideration are unramified at every non-archimedean place.

Second, we consider the effective multiplicity one problem, determining the minimal number of Hecke eigenvalues needed to distinguish two Maass--Hecke cusp forms with the same Laplace eigenvalue. We prove that for any fixed $\eta>0$, if two Maass--Hecke cuspforms, with eigenparameter $t$, share Hecke eigenvalues $\lambda_{\phi_1}(n) = \lambda_{\phi_2}(n)$ for all $n < \eta t$, and $t$ is sufficiently large, then the forms are proportional. This improves the previously known best bound due to Huntley \cite{hunt}. 
Key ingredient for the improvement is the result by Brook and Lindenstrauss \cite{MR3260861} that classifies quantum limits of a joint eigenfunction of a Hecke operator and the Laplace--Beltrami operator on arithmetic hyperbolic surfaces.

We also discuss generalizations of these results to Maass--Hecke cuspforms on $\GL_2$ over arbitrary number field.
\end{abstract}
\maketitle

\section{Introduction}

Let $\X = \PGL_2(\Z)\backslash \HH$ be the modular surface and let $\Delta$ be the Laplace--Beltrami operator on $\X$. The space of square-integrable functions $L^2(\X)$ decomposes as a direct sum: $L^2(\X) = \Theta \oplus \Theta_0 \oplus L_{\text{cusp}}^2(\X)$, where $\Theta$ is the space spanned by incomplete Eisenstein series, $\Theta_0$ is the space of constant functions, and $L_{\text{cusp}}^2(\X)$ is the space of cuspidal functions.  Selberg's celebrated theorem \cite{MR0088511} establishes that the spectrum of $-\Delta$ on $\Theta_0 \oplus L_{\text{cusp}}^2(\X)$ is discrete, consisting of infinitely many eigenvalues in $[0, \infty)$ tending to $+\infty$.  The spectrum of $-\Delta$ on $\Theta$ is purely continuous.

The Hecke operators $\{T_n\}_{n \geq 1}$ form a commuting family of self-adjoint operators on $\X$, defined by
\begin{equation*}
T_nf(z) = \frac{1}{\sqrt{n}}\sum_{ad=n} \sum_{b \bmod d} f\left(\frac{az+b}{d}\right).
\end{equation*}
These operators commute with $\Delta$, and consequently, $L_{\text{cusp}}^2(\X)$ admits a basis of joint eigenfunctions of $\Delta$ and $\{T_n\}_{n\geq 1}$, known as Maass--Hecke cusp forms.  For a Maass--Hecke cusp form $\phi$, we denote its eigenvalue of the Laplace--Beltrami operator by $\lambda_\phi = \frac{1}{4} + t_\phi^2$, where $t_\phi$ is called the eigenparameter of $\phi$.  The $n$-th Hecke eigenvalue of $\phi$ is denoted by $\lambda_\phi(n)$.

The strong multiplicity one theorem \cite{MR546599,MR618323,MR623137} guarantees that two Maass--Hecke cusp forms are proportional if they share the same Laplacian eigenvalue and Hecke eigenvalues for almost all primes.  This raises a natural question, analogous to Linnik's problem for Dirichlet characters: given two Maass--Hecke cusp forms $\phi_1$ and $\phi_2$ with the same Laplacian eigenvalue, how large can the smallest prime $p$ be such that their $p$-th Hecke eigenvalues differ (i.e., $\lambda_{\phi_1}(p) \neq \lambda_{\phi_2}(p)$)?  Alternatively, following Linnik's approach as described in \cite[\S1]{DK}, one can ask about the number of Maass--Hecke cusp forms sharing a given Laplacian eigenvalue and a set of Hecke eigenvalues.

This article addresses these questions in two concrete forms:
\begin{enumerate}
    \item What is the dimension of the joint eigenspace of a given finite set of Hecke operators and the Laplacian operator $\Delta$?
    \item What is the minimal number of Hecke eigenvalues required to determine the proportionality of two Maass--Hecke cusp forms?
\end{enumerate} 

In regards to the first question, we prove that only a relatively small number of Hecke operators are required to give a very strong upper bound for the dimension of the joint eigenspaces.
\begin{theorem}\label{thm:linnik}
Fix any $\alpha>0$. Among Maass--Hecke cuspforms having the eigenparameter in the interval $\lbrack T,T+1\rbrack$, the dimension of the joint eigenspace the Hecke operators $T_p$ with $p$ less than $(\log T)^\alpha$ is
\[
O(T^{\frac{4}{\alpha}+\epsilon}).
\]
\end{theorem}

\begin{remark}
For a fixed finite set of primes $S$, it was shown in \cite{sr} that the dimension of a joint eigenspace of $\Delta$ and the Hecke operators $\{T_p\}_{p \in S}$ with eigenparameter $t$ is $\ll \frac{t}{(\log t)^{|S|+1}}$.
\end{remark}

\begin{remark}
For general closed Riemannian surfaces, an eigenspace of $\Delta$ with the eigenparameter $t$ has dimension $\ll t$ \cite{Lars}, and this is sharp for the round sphere. For negatively curved surfaces, a stronger upper bound of the form $\ll \frac{t}{\log t}$ is proven by B\'erard \cite{be77}.
\end{remark}

The second question is to determine how many Hecke operators are required to actually single out an automorphic form,
often referred to as an effective multiplicity one theorem \cite{MR778093,Bru06}. We provide a slight improvement over the previously known bound \cite{hunt}.

\begin{theorem}\label{thm:effective_Q}
Fix $\eta>0$.
There exists $T_\eta>0$ such that for any two Maass--Hecke cusp forms $\phi_1$, $\phi_2$
with the same Laplace eigenpapameter $t>T_{\eta}$
such that their Hecke eigenvalues also satisfy
\begin{equation*}
\lambda_{\phi_1}(n) = \lambda_{\phi_2}(n)
\end{equation*}
for all $n <\eta t$ implies that $\phi_1$ and $\phi_2$ are equal (up to constant multiplication).
\end{theorem}

Prior work established the theorem only for $\eta$ exceeding some fixed positive constant \cite{hunt}. However, it is known that any subconvexity estimate for the Rankin-Selberg $L$-function associated with two Maass--Hecke cusp forms would yield a power-saving improvement over the number of Hecke eigenvalues required by Theorem~\ref{thm:effective_Q}.  Specifically, under the Generalized Lindelöf Hypothesis for these $L$-functions, $O_\epsilon(t^\epsilon)$ Hecke eigenvalues would suffice.  Assuming the stronger Grand Riemann Hypothesis, this number reduces to $O((\log t)^2)$, as shown in \cite{sr}.  The strongest form of related conjectures predicts that equality of the eigenparameters, $t_{\phi_1} = t_{\phi_2} = t$, alone should imply that $\phi_1$ and $\phi_2$ are proportional \cite{MR489542}.
\begin{remark}
For holomorphic cuspforms of level $1$, a conjecture by Hida and Maeda \cite{MR1610859} implies that the second Hecke eigenvalue of a holomorphic Hecke cuspform uniquely determines the form.   
\end{remark}

\subsection{Generalization to Maass forms on \texorpdfstring{$\PGL_2$}{GL2} over number fields}

In this article, we also study the generalization of Theorem \ref{thm:linnik} and Theorem \ref{thm:effective_Q} to automorphic forms on $\PGL_2$ over arbitrary number fields. 

Let $F$ be a number field. 
For each place $v$ of $F$, let $F_v$ denote the completion of $F$ at $v$ with respect to the normalized valuation $\|\cdot\|_v$ and $\A = \A_F=\prod_v F_v$ be the ring of ad\'eles of $F$.
Let $d_F$ be the number of archimedean places of $F$.
For non-archimedean place $v<\infty$, we let $\mfo_v$ denote the ring of integers of $F_v$,
$\mfp_v$ the unique maximal ideal in $\mfo_v$
and $q_v$ the cardinality of $\mfo_v/\mfp_v$.

For each place $v$ of $F$, let $\K_v$ be the connected maximal compact subgroup of $\PGL_2(F_v)$:
\[\K_v = \begin{cases}
\PSO(2) & \text{ if } F_v=\R, \\
\PU(2) & \text{ if } F_v=\C, \\
\PGL_2(\mfo_v) & \text{ if } v< \infty
\end{cases}
\]
and $\K=\prod_v \K_v$. 
Then a cuspidal automorphic representation $(\pi, V_\pi)$ of $\PGL_2(F)\backslash \PGL_2(\A)$
is called spherical if $\dim(\pi^{\K})=1$.

Let $\pi\cong \otimes_v \pi_v$ be a cuspidal automorphic spherical representation of $\PGL_2(\A)$. 
Let $\{t_{\pi_v}\}_{v\mid \infty}$ be the Langalnd parameter, i.e., $\lambda_{\pi_v} = \frac{1}{4}+t_{\pi_v}^2$ when $F_v=\R$ and $\lambda_{\pi_v} = 1+4t_{\pi_v}^2$ when $F_v=\C$ are eigenvalues of the Laplace--Beltrami operator on $\X_F= \PGL_2(F) \backslash \PGL_2(\A)/\maxK$ corresponding to the automorphic function corresponding to $\pi$ (see \S\ref{ss:notation}). 
Let $t_\pi^2 = \lambda_\pi = \sum_{v|\infty}\lambda_{\pi_v}$, and let $d$ be the dimension of $\X_F$, i.e., 
\[
d= 2\#\{v~:~F_v =\mathbb{R}\}+3\#\{v~:~F_v =\mathbb{C}\}.
\]

Then the generalization of Theorem \ref{thm:linnik} to the number field $F$ we prove in the article is:
\begin{theorem}\label{thm:linnikF}
For $1<G<T$, let $\scrA(T, G)$ be the set of irreducible cuspidal representations of $\PGL_2(\A)$ satisfying the followings:
$\pi=\otimes_v \pi_v$,  
\begin{itemize}
\item $\pi_v$ is unramified for all places $v$ of $F$;
\item $T\leq t_\pi \leq T+G$.
\end{itemize}
For a cuspidal automorphic representation $\pi_0$ in $\scrA(T, G)$, let 
\begin{equation*}
\scrA_{\pi_0} (T, G; \alpha) = \left\{\pi=\otimes_v \pi_v \in \scrA(T, G):\, \pi_v\cong \pi_{0, v} \text{ for any }v\in S_{\finite} \text{ such that } q_v \leq (\log T)^\alpha \right\}. 
\end{equation*}
Then 
\begin{equation*}
\#\scrA_{\pi_0}(T, 1, \alpha) \ll_{\varepsilon, F} 
T^{\frac{d-1+3\#S_\infty}{\alpha}+\varepsilon}.
\end{equation*}
\end{theorem}

We also establish the following generalization of Theorem \ref{thm:effective_Q} subject to Conjecture \ref{aque2}:
\begin{theorem}\label{thm:QUE-multone}
Assume Conjecture \ref{aque2}. Fix $\eta>0$ and $B>0$.
Take any irreducible cuspidal representations $\pi_1=\otimes_{v} \pi_{1, v}$, $\pi_2=\otimes_v \pi_{2, v}$ of $\PGL_2(\A)$ such that $\pi_{1, v}$ and $\pi_{2, v}$ are unramified for all places $v$ of $F$. 
We further assume that $\pi_{1, v}\cong \pi_{2, v}$ for all unramified places $v$ and 
let $t_v = t_{\pi_{1, v}}$ for $v\in S_{\infty}$. 
Assume that there exists $t>0$ such that $\frac{|t_v|}{t}\in [B^{-1}, B]$ for any $v\in S_\infty$.
We further assume that $\pi_{1, v}\cong \pi_{2, v}$ for non-archimedean places $v$ such that $N(q_v) \leq \eta t$. 

Then there exists $t_\eta>0$ such that if $t>t_\eta$ then $\pi_1$ and $\pi_2$ are isomorphic. 
\end{theorem}

We note that Theorem \ref{thm:linnik} and Theorem \ref{thm:effective_Q} follow from Theorem \ref{thm:linnikF} and Theorem \ref{thm:QUE-multone} repectively by taking $F=\mathbb{Q}$.

\subsection{Sketch proofs}

The proofs of Theorem \ref{thm:linnik} and its generalization, Theorem \ref{thm:linnikF}, employ distinct methods from those used in the proofs of Theorem \ref{thm:effective_Q} and its generalization, Theorem \ref{thm:QUE-multone}: spectral large sieve inequalities and quantum unique ergodicity, respectively.
Consequently, we present separate proof sketches for each pair of theorems.

\subsubsection{Spectral large sieve inequalities}
We first recall a result similar to Theorem \ref{thm:linnik} obtained in \cite[Theorem 1]{DK} for the modular forms of the elliptic curves.
\begin{theorem}[\cite{DK}]\label{thm:DK}
The number of modular forms of weight $2$ corresponding to elliptic curves on $\mathbb{Q}$ of the conductor less than $Q$ which for every prime $p \leq (\log Q)^\alpha$ have a fixed number of points modulo $p$ is
\[
O_\epsilon(Q^{\frac{8}{\alpha}+\epsilon}).
\]
\end{theorem}
To illustrate the idea of the proof of the theorem, let $S(Q)$ be the set of elliptic curves that have a conductor less than $Q$. 
Let us denote by $\lambda_E(n)$ the normalized $n$-th Hecke eigenvalue of the modular form corresponding to the elliptic curve $E$. 
Then, for a prime $p$, we have $1=\lambda_E(p)^2-\lambda_E(p^2)$ which implies that $\lambda_E(p)$ and $\lambda_E(p^2)$ cannot both be simultaneously close to $0$. Fix an elliptic curve $\tilde{E}$, and assume for simplicity that the Hecke eigenvalues $\{\lambda_{\tilde{E}}(p)\}_{p<X}$ are bounded away from $0$ by some positive constant. Say we are given an estimate of the form
\begin{equation} \label{eq:amp}
\sum_{E\in S(Q)} \left|\sum_{n<X} x_n \lambda_E(n)\right|^2 \ll  Q^\beta \sum_{n<X}|x_n|^2     
\end{equation}
for some $\beta>0$ that holds for arbitrary set of complex numbers $\{x_n\}_{n<X}$. Then denoting by $M(l,X)$ the set of squarefree integers less than $X$ having at most $l$ prime factors, one may choose $x_n = \lambda_{\tilde{E}}(n)$ for $n\in M(l,X)$, and $x_n=0$ otherwise, to obtain
\[
\#\{E\in S(Q)~:~ \lambda_E(p) = \lambda_{\tilde{E}}(p),~p<X\}\ll Q^\beta \left(\sum_{n\in M(l,X)}|\lambda_{\tilde{E}}(n)|^2 \right)^{-1}
\]
from the multiplicativity of $\lambda_E(n)$ and the positivity of the summands on the left hand side of \eqref{eq:amp}. Because we assumed that $\{\lambda_{\tilde{E}}(p)\}_{p<X}$ are bounded away from $0$ by some positive constant, say $c>0$, we may use the multiplicativity of the Hecke eigenvalues in order to bound the right hand side in terms of $l,X,Q,$ and $c$. By appropriately choosing these parameters, we obtain an upper bound for
\[
\#\{E\in S(Q)~:~ \lambda_E(p) = \lambda_{\tilde{E}}(p),~p<X\},
\]
which is the quantity appearing in Theorem \ref{thm:DK}. 

Estimates of the form \eqref{eq:amp} are called \textit{large sieve inequalities}, and \cite{DK} studies large sieve inequalities for the family of $\GL_n$ automorphic forms over $\mathbb{Q}$ that are identical at the archimedean places in the conductor (non-archimedean) aspect. 

This paper focuses on spectral large sieve inequalities, the large sieve inequalities that consider the weight aspect of families of automorphic forms with a fixed conductor. 
For $\GL_n$ automorphic forms over arbitrary number field $F$ that are unramified at every non-archimedean place, we prove Theorem \ref{thm:spec_large_sieve_n}.
For the case when $n=2$, 
we establish a spectral large sieve inequality for their symmetric squares (Theorem \ref{thm:spec_large_sieve_n=2_sym2}). 
Combining this with Theorem \ref{thm:spec_large_sieve_n}, we obtain Theorem \ref{thm:linnikF}.

Note that when $F=\mathbb{Q}$, we have the following spectral large sieve inequality for symmetric squares of Maass--Hecke cuspforms \cite{MR4635352}. 

\begin{theorem}[\cite{MR4635352}]\label{sieve2}
Let $\{\phi_j\}$ be the complete set of Maass--Hecke cuspforms on $\mathbb{X}$ and let 
\[
L(s,\mathrm{sym}^2, \phi_j) = \sum_{n=1}^\infty b_j(n) n^{-s}.
\]
Then, for $1 \leq G \leq T$, we have
\[
\sum_{T \leq t_j \leq T+G} \left|\sum_{n \leq N} x_n b_j(n) \right|^2 
\ll_\epsilon T^\epsilon N^\epsilon
\begin{cases}
NT^{1/2}+TG &N\leq T\\
GN+N^{3/2}& T\leq N\leq T^2\\ 
N^2 T^{-1} & T^2\leq N \end{cases} .
\]
\end{theorem}
Theorem \ref{thm:spec_large_sieve_n=2_sym2} when $\sigma=0$ and $F=\mathbb{Q}$ reads:
\begin{equation}\label{siiieve}
\sum_{T \leq t_j \leq T+G} \left|\sum_{n \leq N} x_n b_j(n) \right|^2 \ll_\epsilon T^\epsilon N^\epsilon (N+T^{4}G^{2}) \sum_{n \leq N} |x_n|^2.
\end{equation}
While Theorem \ref{sieve2} is stronger than this estimate in a large range $N<T^\frac{5}{2}G$, it does not improve Theorem \ref{thm:linnik}, because the saving we get from the arithmetic argument is by the factor of $\sim N^{1-\frac{1}{\alpha}}$. 

Also, note that the method of \cite{DK} yields
\begin{equation}\label{siiieve2}
\sum_{T \leq t_j \leq T+G} \left|\sum_{n \leq N} x_n b_j(n) \right|^2 \ll_\epsilon T^\epsilon N^\epsilon (N+N^{1/2}T^{5/2}G^{3/2}) \sum_{n \leq N} |x_n|^2,
\end{equation}
Note that \eqref{siiieve} is stronger than \eqref{siiieve2} in the range $TG \ll N\ll T^5G^3$. For instance, if we use \eqref{siiieve2} instead of \eqref{siiieve}, then the upper bound we get for Theorem \ref{thm:linnik} is
\[
O(T^{\frac{5}{\alpha}+\epsilon}),
\]
which is slightly weaker than what we prove. 

The improvement of \eqref{siiieve} over \eqref{siiieve2} comes from the fact that we are exploiting the assumption of $\pi$ being unramified everywhere (or $\phi$ being a form on $\PGL_2(\mathbb{Z})\backslash \mathbb{H}$). 
This makes the computation involving the corresponding automorphic $L$-functions simpler and nicer, and that allows us to make use of the region $0<\mathrm{Re}(s)<\frac{1}{2}$ beyond the critical line.

\subsubsection{Quantum Unique Ergodicity and multiplicity one}
In the subsequent sections, we are going to present only the proof for Theorem \ref{thm:QUE-multone}, as Theorem \ref{thm:effective_Q} is a specialization of the theorem to $F=\mathbb{Q}$. (Note that Conjecture \ref{aque2} when $F=\mathbb{Q}$ is \cite{MR3260861}.) However, for this section, we illustrate the idea of the proof for Theorem \ref{thm:effective_Q}. A key ingredient in the proof of Theorem \ref{thm:effective_Q} is the equidistribution theorem of Brook and Lindenstrauss  \cite{MR3260861} which recovers the result of \cite{lin} using only finitely many Hecke operators, which we state as follows:
\begin{theorem}[\cite{MR3260861}]\label{linalt}
Fix a finite set of primes $S$ and let $\{f_j\}$ be a sequence of $L^2(\X)$ normalized joint eigenfunctions of $\Delta$ and $\{T_p\}_{p\in S}$ with strictly increasing Laplacian eigenvalues. Then any weak-$\ast$-limit of
\[
|f_j (z)|^2d\mu(z)
\]
is of the form
\[
cd\mu(z)
\]
for some constant $c\in [0,1]$.
\end{theorem}
Note that we are using only finitely many Hecke operators here, which is critical for the proof to work. For our application, we need to exclude the possibility that all the mass escapes to the non-compact part of $\X$.
\begin{lemma}[Proposition \ref{prop:escape} when $F=\mathbb{Q}$]\label{lem1}
Fix $\eta>0$. Let $\{f_j\}$ be a sequence of $L^2(\X)$ normalized functions such that $f_j$ is a joint eigenfunction of $\Delta$ and $\{T_n\}_{n<N_j}$ with the eigenparameter $t_j$, and $N_j=\eta t_{j}$. Here we assume that $t_j \to +\infty$ as $j \to +\infty$. Then any weak-$\ast$-limit of
\[
|f_j (z)|^2d\mu(z)
\]
is of the form
\[
cd\mu(z)
\]
for some constant $c\in (0,1]$.
\end{lemma}
\begin{remark}
When $f_j$ is a joint eigenfunction of $\Delta$ and \textit{all} Hecke operators, i.e., if $f_j$ is a Maass--Hecke cuspform, it is known that the only possible value of $c$ is $1$ \cite{sou}. When $F$ is an arbitrary number field, this is proven in \cite{Z10}.
\end{remark}
To prove Theorem \ref{thm:effective_Q}, we proceed by contradiction. Suppose there exist infinitely many eigenparameters $t_j$ such that there are at least two distinct cusp forms with the same Hecke eigenvalues $\lambda_\phi(n)$ for all $n < \eta t_j$. Since the Fourier coefficients at $i\infty$ of a Maass--Hecke cusp form are proportional to its Hecke eigenvalues, we can construct a non-zero cusp form by taking a linear combination of these two cusp forms, such that its Fourier coefficients vanish for all $n < \eta t_j$. This linear combination remains a joint eigenfunction of $\Delta$ and $T_n$ for $n < \eta t_j$, and thus its $L^2$ mass should equidistribute as $t_j \to \infty$ according to Theorem \ref{linalt}. However, the vanishing Fourier coefficients for $n < \eta t_j$ imply that this function has arbitrarily small mass in a compact region determined by $\eta$, contradicting Lemma \ref{lem1}.

\section{Preliminaries}\label{ss:notation}

Let $F$ be a number field and $\bo_F$ its ring of integers. 
For each place $v$ of $F$, we denote $F_v$ the completion of $F$ at $v$ with respect to the normalized valuation $|\cdot|_v$. 
We reserve $|\cdot|$ for the uaual absolute value of real or complex numbers. 
Let $S_F$ denote the set of all places of $F$, 
$S_\infty$ the set of all archimedean places of $F$ 
and $S_{\finite}$ be the set of all non-archimedean places of $F$. 
We further separate $S_\infty = S_{\R}\cup S_{\C}$ where 
\[S_{\R} = \{v\in S_\infty:\, F_v=\R\}\quad \text{ and } \quad S_{\C} = \{v\in S_{\infty}:\, F_v=\C\}.\]
We give explicit description of the absolute valuation $|\cdot|_v$ for each place $v$. 
For $v\in S_\infty$, 
\[|x|_v = |x|^{\epsilon_v+1} \quad\text{ where } \epsilon_v = \begin{cases}0 & \text{ if } F_v=\R, \\ 1 & \text{ if } F_v=\C.\end{cases}
\]
For $v\in S_{\finite}$, we let $\bo_v$ denote the ring of integers of $F_v$, $\bp_v$ the unique maximal ideal and $q_v$ the cardinality of $\bo_v/\bp_v$. 
We let $\varpi_v\in \bo_v$ be the uniformizer so $\bp_v=\left<\varpi_v\right>$ and 
\[|\varpi_v|_v = q_v^{-1}.\]
We let $\bo_v^\times=\{u\in \bo_v:\, |u|_v=1\}$ be the multiplicative group of units in $\bo_v$.

We let $\A=\A_F=\prod_v' F_v$ (with respect to $\bo_v$ for $v\in S_{\finite})$ be the ring of ad\'eles over $F$
and $\A^\times = \A_F^\times=\prod_v' F_v^\times$ (with respect to $\bo_v^\times$ for $v\in S_{\finite}$) be the multiplicative group of id\'eles of $F$. 
We further let $F_\infty=\prod_{v\in S_\infty} F_v$ and $\A_{\finite} = \prod_{v\in S_{\finite}}' F_v$ be the ring of finite ad\'eles of $F$.
We regard $F$ as a subring of $\A$. 
Following \cite{gj}, we define 
\begin{equation*}
F_\infty^+ = \left\{y\in \A^\times:\, y_v=u\text{ for all } v\in S_\infty \text{ for some } u>0 \text{ and } y_v=1\text{ for all } v\in S_{\finite}\right\}
\end{equation*}
and 
\begin{equation*}
\A^1 = \left\{y\in \A^\times :\, \|y\|=1\right\}.
\end{equation*}
Then we have the isomorphism 
$F^\times\backslash \A^\times 
\cong F_\infty^+\cdot (F^\times\backslash\A^1)$.

For $x=\{x_v\}_{v}=\{x_v\}_{v\in S_F}\in \A$, we denote by $\|x\| = \prod_v |x_v|_v$ 
the canonical absolute value on $\A^\times$. 
This absolute value is normalised, so that $\|\alpha\|=1$ for any $\alpha\in F^\times$. 
For $x\in F_\infty$, we let $\|x\|_\infty = \prod_{v\in S_\infty} |x_v|_v$ 
and for $x\in \A_{\finite}$, we let $\|x\|_{\finite} = \prod_{v\in S_{\finite}} |x_v|_v$. 

For any $\alpha\in \A^\times$, we write $(\alpha)$ to denote the fractional ideal 
\begin{equation*}
(\alpha) = \prod_{v< \infty} (\bp_v\cap F)^{\ord_v(\alpha_v)}.
\end{equation*}
Let $h$ be the order of the class group of $F$. 
We fix $t_1, \ldots, t_h\in \A_{\finite}^\times$ such that the fractional ideals $(t_j) = \prod_{v<\infty}(\bp_v\cap F)^{\ord_v(t_{j, v})}$ represent the ideal classes of $F$. 

Following \cite[\S1]{BK11}, we choose additive and multiplicative Haar measures for each place $v$ in the following way. 

We fix an additive character $\psi=\prod_v \psi_v$ of $F\backslash \A$ whose conductor is the inverse different $\bd^{-1}$ of $F$ so $\psi = \psi_{\Q} \circ \tr$ where $\psi_\Q$ is the additive character of $\Q\backslash \A_{\Q}$ which is unramified at all finite primes and whose restriction to $\R$ is the exponential function $e(x) = e^{2\pi ix}$ and $\tr:\A_F\to \A_\Q$ is the trace map. 

For $v\in S_{\infty}$, we let $d_vx$ be the ordinary Lebesgue measure 
and $d_v^\times x$ be the multiplicative Haar measure on $F_v^\times$ such that 
\[d_v^\times x = \begin{cases}
\frac{d_vx}{2|x|_v} & \text{ when } v\in S_{\R}, \\
\frac{d_vx}{\pi |x|_v} = \frac{d_vx}{\pi|x|^2} & \text{ when } v\in S_{\C}. 
\end{cases}\]
Then the Mellin inversion formula at an archimedean place $v$ takes the form 
\[\tilde{f}(s, \epsilon) = \int_{F_v^\times} f(x)\sgn^\epsilon(x) |x|_v^{s-\frac{1}{2}} \, d_v^\times x
\iff f(x) = \sum_{\epsilon\in \{0, 1\}} \frac{1}{2\pi i } \int_{(\sigma)} \tilde{f}(s, \epsilon) \sgn^\epsilon(x) |x|_v^{-s+\frac{1}{2}} \, ds \]
when $F_v=\R$ and 
\[\tilde{f}(s, k) = \int_{F_v^\times} f(x) (x|x|^{-1})^{k} |x|_v^{s-\frac{1}{2}} \, d_v^\times x\iff f(x) = \sum_{k\in \Z} \frac{1}{2\pi i} \int_{(\sigma)} \tilde{f}(s, k) (x|x|^{-1})^{-k} |x|_v^{-s+\frac{1}{2}} \, ds\]
when $F_v=\C$.
Here each $s$-integral is taken along a vertical line $\Re(s)=\sigma$ for $\sigma\in \R$, to the right of any poles of the integrand. 

For $v\in S_{\finite}$, we let $d_vx$ be the additive Haar measure on $F_v$ such that 
\[\vol(\bo_v) = \int_{\bo_v}1\, d_vx = 1.\]
We define the multiplicative Haar measure on $F_v^\times$ as 
\[d_v^\times x = \frac{d_vx}{(1-q_v^{-1})|x|_v}.\]
Then we get
\[\int_{\bo_v^\times} 1\, d_v^\times x = 1.\]

On $\A$, the additive Haar measure is given as $dx = \prod_v d_vx_v$ which is self-dual, 
and the multiplicative Haar measure on $\A^\times$ is given as $d^\times x = \prod_v d_v^\times x_v$.

Let $\G$ denote the group $\GL_2$. 
The center $\CentZ$ of $\G$ is $\CentZ = \left\{\sm a & \\ & a\esm\right\}$ and we set $\bar{\G} = \CentZ \backslash \G\cong \PGL_2$. 
We also introduce the subgroups: 
\begin{equation*}
\BorelB = \left\{\bpm * & * \\ 0 & * \ebpm \right\}
\quad\text{ and } \quad 
\N = \left\{\bpm 1 &* \\ 0 & 1\ebpm \right\}.
\end{equation*}
We also write, $n(x) = \sm 1 & x\\ 0 & 1\esm \in \N$ and $a(y) = \sm y & \\ & 1\esm$. 
By the Iwasawa decomposition, any element $g\in \bar{\G}(F_v)$ can be represented as 
$n(x)a(y)\kappa$ for $x\in F_v$, $y\in F_v^\times$ and $\kappa\in \K_v$ for each place $v$ of $F$.
We can write a Haar measure $d_vg$ on $\bar{\G}(F_v)$ as
\[\int_{\bar{\G}(F_v)} f(g)\, d_v g
= \int_{\K_v} \int_{F_v^\times} \int_{F_v} f(n(x) a(y) \kappa) \, d_v x\frac{d_v^\times y}{|y|_v} d_v\kappa.\]
Similarly we select measures on $\bar{\G}(\A)$ and $\maxK$ such that 
\begin{equation*}
\int_{\bar{\G}(\A)} f(g)\, dg = \int_{\maxK} \int_{\A^\times} \int_{\A} f(n(x)a(y)\kappa)\, dx\, \frac{d^\times y}{\|y\|} \, d\kappa.
\end{equation*}
Here $d\kappa=\prod_v d_v\kappa_v$.

\subsection{Automorphic functions on \texorpdfstring{$\bar{\G}(F)\backslash \bar{\G}(\A)$}{G(F)G(A)}}

Let $\X_F = \bar{\G}(F)\backslash \bar{G}(\A)/\maxK$. 
We let $L^2(\X_F)$ be the space of functions $f$ on $\X_F$ such that 
\begin{equation*}
\|f\|_2^2 = \int_{\bar{\G}(F)\backslash \bar{\G}(\A)}|f(g)|^2\, dg < +\infty. 
\end{equation*}
We say that $f$ is cuspidal if 
\begin{equation*}
\int_{F\backslash \A} f(n(x)g) \, dx = 0.
\end{equation*}
Then $f$ has the following Fourier expansion: 
\begin{equation*}
f(g) = \sum_{\alpha\in F^\times} \widehat{f}\left(\bpm \alpha & \\ & 1\ebpm g\right) 
\end{equation*}
where 
\begin{equation*}
\widehat{f}(g) = \int_{F\backslash \A} f(n(x) g) \psi(-x)\, dx.
\end{equation*}
By the Iwasawa decomposition, for $g=n(x)a(y)\kappa$ with $x\in \A$, $y\in \A^\times$ and $\kappa\in \maxK$, since $f$ is right $\maxK$-invariant
\begin{equation*}
\widehat{f}(n(x)a(y)\kappa) = \psi(x)\widehat{f}(a(y)). 
\end{equation*}

We define the Laplace--Beltrami operators $\Delta_{F_v}$.  
On $\H^2=\bar{\G}(\R)/\PO(2)$
\begin{equation*}
\Delta_{\R} = y^2\left(\frac{\partial^2}{\partial x^2}+\frac{\partial^2}{\partial y^2}\right)
\end{equation*}
with the parameters $g=n(x)a(y)\kappa$ for $x\in \R$, $y>0$ and $\kappa\in \PO(2)$.
On $\H^3=\bar{\G}(\C)/\PU(2)$, 
\begin{equation*}
\Delta_{\C} = y^2\left(\frac{\partial^2}{\partial x_1^2}+\frac{\partial^2}{\partial x_2^2} + \frac{\partial^2}{\partial y^2}\right) - y\frac{\partial}{\partial y}, 
\end{equation*}
with the parameters $g=n(x_1+ix_2)a(y) \kappa$ for $x_1, x_2\in \R$, $y>0$ and $\kappa\in \PU(2)$.

Assume that $f$ is an eigenfunction of Laplace--Beltrami operators for all archimedean places of $F$: for any archimedean place $v$, 
\begin{equation*}
-\Delta_{F_v} f = \lambda_v(f) f 
\end{equation*}
and we write
\begin{equation*}
\lambda_v(f) = (1+\epsilon_v)^2 \left(\frac{1}{4}+t_{f, v}^2\right). 
\end{equation*}
Then the Laplace--Beltrami operator $\Delta$ on $\X_F$ is
\[
\Delta = \sum_{v\in S_\infty } \Delta_{F_v}
\]
and the eigenvalue corresponding to $f$ is $\lambda_f =  \sum_{v\in S_\infty } \lambda_v(f) $. We let $t_f\geq 0$ be defined by $\lambda_f = t_f^2$.
When $f$ is cuspidal and eigenfunction of Laplace-Beltrami operator $\Delta$ then we say that $f$ is a Maass cusp form. 
Following \cite[\S4.2]{BK11}, we have 
\begin{equation*}
\widehat{f}(a(\{y_\infty, y_{\finite}\})) 
= \int_{F\backslash \A} f(n(x)a(\{y_\infty, y_{\finite}\}))\psi(-x) \, dx
= W_{f, \infty}(y_\infty)\cdot \rho_{f}(y_{\finite})
\end{equation*}
for $y=\{y_\infty, y_{\finite}\}\in \A^\times$ such that $y_{\infty}\in F_\infty^\times$ and $y_{\finite}\in \A_{\finite}^\times$.
Here $\rho_f(y_{\finite})\in \C$ and 
$W_{f, \infty} = \prod_{v\mid \infty}W_{f, v}$, such that for $y\in F_v^\times$ for each archimedean place $v$, 
\begin{equation*}
W_{f, v} (y) = \begin{cases}
\sqrt{|y|} K_{it_{f, v}}(2\pi |y|) & \text{ when } F_v=\R, \\
|y|K_{2it_{f, v}}(4\pi|y|) & \text{ when } F_v=\C. 
\end{cases}
\end{equation*}

Take $u\in \A$ such that $n(u)\in \maxK$, which implies that $u=\{u_v\}_v$, $u_v\in \bo_v$ for all non-archimedean place $v$ and $u_v=0$ for all archimedean place $v$. 
For any $y=\{y_v\}\in \A^\times$, 
\begin{equation*}
\widehat{f}(a(y)) = \widehat{f}(a(y)n(u)) = \widehat{f}(n(yu) a(y))
= \psi(yu) \widehat{f}(a(y)).
\end{equation*}
So $\widehat{f}(a(y))\neq 0$ implies that $y_v\in \bd_v^{-1}$. 
Therefore $\rho_f(y)=\rho_f(y_{\finite})=0$ unless $(y) = (y_{\finite})\subset \bd_F^{-1}$. 
Then we get 
\begin{equation*}
f(n(x)a(y)\kappa)
= \sum_{\substack{\alpha\in F^\times\\ (\alpha y)\in \bd^{-1}}} \rho_f(\alpha y) \psi(\alpha x) W_{f, \infty}(\alpha y_\infty).
\end{equation*}
Taking $x=0$, by \cite[(4.7)]{BK11}, we get
\begin{equation*}
f(a(y)\kappa)
= \sum_{\alpha\in \bo_F^\times\backslash ((y)^{-1} \bd^{-1} \cap F^\times)} \rho_f(\alpha y) 
\sum_{\eta\in \bo_F^\times} W_{f, \infty}(\eta \alpha y_\infty).
\end{equation*}

For a given non-archimedean place $v$, we define the Hecke operator at $v$ following \cite[\S4.6]{Bump97}. 
For a non-zero integer $k$, let $\bM_{v}(k)$ be the subset of $2\times 2$ matrices $g$ with coefficients in $\bo_v$ such that $|\det g|_v = q_v^{-k}$. 
By the Cartan decomposition, we have 
\[\bM_v(k) = \coprod_{\substack{k_1, k_2\geq 0\\ k_1+k_2=k}} \GL_2(\bo_v) \bpm \varpi_v^{k_1} & \\ & \varpi_v^{k_2}\ebpm \GL_2(\bo_v).\]
We define the Hecke operator $T(\bp_v^k)$ by 
\[(T(\bp_v^k) f)(g_1) = \int_{F_v^\times\backslash (F_v^\times\bM_v(k))} f(g_1g)\, d_vg\quad\text{ for } g_1\in \G(\A).\]
Here the multiplication occurs as $gg_1 = \{(g_1g)_w\}_w$ where $(g_1g)_w = g_{1, w}$ for $w\neq v$ and $(g_1g)_v=g_{1, v}g$. 
If $f$ is an eigenfunction of Hecke operator, we have 
\[T(\bp_v^k) f = q_v^{\frac{k}{2}}\lambda_f(\bp_v^k) \cdot f.\]

When $\pi=\otimes_v \pi_v$ is an irreducible cuspidal representation of $\PGL_2(\A)$ such that $\pi_v$ is unramified for any place $v$ of $F$, there exists a non-zero cuspidal function $\phi_{\pi}\in L^2(\PGL_2(F) \backslash \PGL_2(\A)/\K)$, uniquely determined up to constant multiplication, such that, for any non-zero integral ideal $\ba$ in $F$, we have 
\[\lambda_{\pi}(\ba) = c\rho_{\phi_{\pi}}(t_j, \gamma)\sqrt{N(\ba)}\]
where $c$ is a constant and $j$, $1\leq j\leq h$ is the unique index such that $\ba = (\gamma)(t_j)\bd$ for some $\gamma\in F$.

For a cusp form $f\in L^2(\PGL_2(F)\backslash \PGL_2(\A)/\K)$, if there exists an irreducible cuspidal representation $\pi$ of $\PGL_2(\A)$ such that $f=c\phi_\pi$ for some constant $c$, then we say that $f$ is a Hecke-Maass cusp form.

\subsection{Cuspidal representations of \texorpdfstring{$\PGL_n$}{GLn}}
Let $\pi=\otimes_v \pi_v$ be cuspidal representation of $\PGL_n(\A)$ and assume that $\pi_v$ are unramified for all places $v$ of $F$.

For each archimdean place $v$ of $F$, let $\{\mu_i(\pi_v)\}_{i=1}^n\in\C^n$ be the Langland parameters for $\pi_v$. 
Similarly, for each non-archimdean place $v$ of $F$, let $\{\alpha_i(\pi)_v\}_{i=1}^{n}\in\C^n$ be the Satake parameters for $\pi_v$. 
Let $\theta_n$ to be the constant towards the Ramanujan-Selberg conjecture: 
\begin{align*}
& |\Re(\mu_i(\pi_v))|\leq \frac{1}{2}-\theta_n \quad\text{ when $v$ is archimedean, } \\
& |\alpha_i(\pi_v)| \leq q_v^{\frac{1}{2}-\theta_n} \quad \text{ when $v$ is non-archimedean}. 
\end{align*}
Ramanujan---Selberg conjecture asserts that $\theta_n<\frac{1}{2}$ can be chose arbitrarily close to $\frac{1}{2}$. 
By \cite{lrs}, we can take 
\begin{equation*}
0\leq \theta_n\leq (n^2+1)^{-1}.
\end{equation*}
When $n=2$, by \cite[Appendix 2]{k03}, we have 
\begin{equation*}
0\leq \theta_2\leq \frac{1}{2}-\frac{7}{64}. 
\end{equation*}

When $v$ is an archimedean place, we define 
\begin{equation*}
L(s, \pi_v) = \prod_{i=1}^n \Gamma_{F_v}\left(s+\mu_i(\pi_v)\right), 
\end{equation*}
where $\Gamma_{\R}(s) = \pi^{-\frac{s}{2}} \Gamma\left(\frac{s}{2}\right)$ and $\Gamma_{\C}(s) = 2(2\pi)^{-s}\Gamma(s)$. 
When $v$ is a non-archimedean place, we define
\begin{equation*}
L(s, \pi_v)= \prod_{i=1}^{n}(1-\alpha_i(\pi_v)q_v^{-s})^{-1}. 
\end{equation*}
The standard $L$-function for $\pi$ is defined by the following Euler product: 
\begin{equation*}
L(s, \pi) = \prod_{v<\infty} L(s, \pi_v) 
= \prod_{v<\infty} \prod_{i=1}^\infty (1-\alpha_i(\pi_v)q_v^{-s})^{-1} 
= \sum_{\vecm} \frac{\lambda_{\pi}(\vecm)}{N(\vecm)^s}
\end{equation*}
which converges absolutely for $\Re(s)>1+\frac{1}{2}-\theta_n$. 
By Rankin-Selberg theory, one can show that in fact the product and series converge absolutely for $\Re(s)>1$. 
The complete $L$-function 
\begin{equation*}
\Lambda(s, \pi) = \prod_v L(s, \pi_v)
\end{equation*}
continues analytically to $s\in \C$ and satisfies the functional equation
\begin{equation*}
\Lambda(s, \pi) = W(\pi)
\Lambda(1-s, \tilde{\pi}), 
\end{equation*}
where $W(\pi)$ is a root number so $|W(\pi)|=1$. Here $\tilde{\pi}=\otimes_v \tilde\pi_v$ is the contragradient representation of $\pi$. 
By \cite{gk}, for archimedean $v$
\begin{equation*}
\{\mu_i(\tilde{\pi}_v)\}_{i=1}^n = \{\overline{\mu_i(\pi_v)}\}_{i=1}^n
\end{equation*}
and for $v<\infty$, 
\begin{equation*}
\{\alpha_i(\tilde{\pi}_v)\}_{i=1}^{n} = \{\overline{\alpha_i(\pi_v)}\}_{i=1}^n.
\end{equation*}
Following \cite{is}, we define the analytic conductor of $\pi$:
\begin{equation*}
C(\pi; t) = 
\prod_{v\in S_\infty}\prod_{i=1}^n (1+|it+\mu_i(\pi_v)|)
\end{equation*}
and let $C(\pi) = C(\pi; 0)$. 

For $i\in \{1, 2\}$, let $\pi_i = \otimes_v \pi_{i, v}$ be cuspidal representations of $\PGL_n(\A)$ with the trivial central character. 
Assume that $\pi_{i, v}$ are unramified for all places $v$. 
We define the Rankin-Selberg convolution for $\pi_1$ and $\pi_2$. 
For an archimedean place $v$, we define 
\begin{equation*}
L(s, \pi_{1, v}\times \pi_{2, v}) = \prod_{i_1=1}^n \prod_{i_2=1}^n \Gamma_{F_v} \left(s+\mu_{i_1}(\pi_v) + \mu_{i_2}(\pi_v)\right).
\end{equation*}
For a non-archimedean place $v$, we define 
\begin{equation*}
L(s, \pi_{1, v}\times\pi_{2, v})= \prod_{i_1=1}^n \prod_{i_2=1}^n \left(1-\alpha_{i_1}(\pi_{1, v}) \alpha_{i_2}(\pi_{2, v}) q_v^{-s}\right)^{-1}.
\end{equation*}
We define 
\begin{equation}\label{e:RS_def}
L(s, \pi_1\times\pi_2) = \prod_{v<\infty} L(s, \pi_{1, v}\times \pi_{2, v}) = \sum_{\vecm} \frac{\lambda_{\pi_1\times\pi_2}(\vecm)}{N(\vecm)^s}
\end{equation}
which converges absolutely for $\Re(s)>2- 2\theta_n$. 
By the Rankin--Selberg integrals \cite{jpss}, one can prove that the Euler product and the series that define $L(s, \pi_1\times \pi_2)$ in \eqref{e:RS_def} converge absolutely for $\Re(s)>1$. 
The completed $L$-function
\begin{equation*}
\Lambda(s, \pi_1\times \pi_2) = L(s, \pi_1\times \pi_2) \prod_{v\in S_\infty}L(s, \pi_{1, v}\times \pi_{2, v}) 
\end{equation*}
continues analytically to $s\in \C$ except a possible simple pole at $s=1$ and $s=0$. 
The completed $L$-function is bounded (away from the poles at $s=1$ and $s=0$) in vertical strips and is of order $1$. 
Note that $\Lambda(s, \pi_1\times \pi_2)$ is entire if and only if $\widetilde{\pi_1}\neq \pi_2$. 
The Rankin-Selberg convolution satisfies the functional equation 
\begin{equation*}
\Lambda(s, \pi_1\times \pi_2) = W(\pi_1\times \pi_2) 
\Lambda(1-s, \widetilde{\pi_1}\times \widetilde{\pi_2}), 
\end{equation*}
where 
$W(\pi_1\times \pi_2)\in \C$ is the root number satisfying $|W(\pi_1\times \pi_2)|=1$.
As before, following \cite{is}, we define the analytic conductor 
\begin{equation}\label{e:C_R-S_def}
C(\pi_1\times \pi_2; t) = 
\prod_{v\in S_\infty} \prod_{i_1=1}^n \prod_{i_2=1}^n (1+|it+\mu_{i_1}(\pi_{1, v})+\mu_{i_2}(\pi_{2, v})|)
\end{equation}
and we let $C(\pi_1\times \pi_2) = C(\pi_1\times \pi_2; 0)$. 
By \cite[(8)]{Bru06}, 
\begin{equation*}
C(\pi_1\times \pi_2; t) \leq \big(C(\pi_1) C(\pi_2) (1+|t|)^{n[F:\Q]}\big)^n.
\end{equation*}

Applying the Phragm\'en-Lindel\"of principle as in \cite[(10)]{Bru06}, we have the following preconvex bound for $-1+2\theta_n \leq \sigma \leq 2-2\theta_n$.
\begin{equation}\label{e:preconvex}
(\sigma-1)^{\delta_{\pi_1, \widetilde{\pi_2}}} L(\sigma+it, \pi_1\times \pi_2) \ll_\varepsilon C(\pi_1\times \pi_2; t)^{l(\sigma)+\varepsilon}, 
\end{equation} 
for $\varepsilon>0$.
Here $l(\sigma)$ is the linear function satisfying $l(-1+2\theta_n)=\frac{3}{2}-2\theta_n$ and $l(2-2\theta_n)=0$, i.e., $l(\sigma) = -\frac{1}{2}\sigma+1-\theta_n$.

Note that for $n=2$ we know the finiteness of the exceptional eigenvalues \cite{CLLL}, hence when $t_{\pi_1}$ and $t_{\pi_2}$ are both sufficiently large, we know that $\pi_1$ and $\pi_2$ are tempered at every Archimedean place. Accordingly, we use the convexity bound \cite{XiaLi10}
\begin{equation}\label{e:preconvexGL2}
(\sigma-1)^{\delta_{\pi_1, \widetilde{\pi_2}}} L(\sigma+it, \pi_1\times \pi_2) \ll_\varepsilon C(\pi_1\times \pi_2; t)^{\frac{1-\sigma}{2}+\varepsilon},
\end{equation} 
for such forms.

\section{Bounding the dimension of a joint eigenspace}

\subsection{Spectral large sieve}\label{ss:spectral_largesieve}

We begin by studying large sieve inequalities in the archimedean aspect.  
While large sieve inequalities in the non-archimedean (level) aspect have been studied previously in \cite{DK}, we adapt their convexity bound based approach to the archimedean setting.

For $Q\gg 1$, let $\scrA_n(\leq Q)$ be a set of irreducible cuspidal representations $\pi=\otimes_v \pi_v$ of $\PGL_n(\A)$ such that $\pi_v$ is unramified for any place $v$ and its analytic conductor $C(\pi)\leq Q$.

Let $\phi$ be a smooth, non-negative and compactly supported function on $[0, +\infty)$ satisfying 
\begin{itemize}
\item $\phi(x)\in [0, 1]$ for any $x\in [0, +\infty)$,
\item $\phi(x)=1$ when $x\in [0, 1]$, and 
\item $\phi(x)=0$ when $x>2$.
\end{itemize}
For $s\in \C$, the Mellin transform of $\phi$ is 
\begin{equation*}
\tilde{\phi}(s) = \int_0^\infty \phi(x)x^s\, \frac{dx}{x}.
\end{equation*}

Following the first part of the proof of \cite[Theorem 3]{Bru06}, by \cite[(15)]{Bru06}, for $X\gg 1$ and $0< \sigma< 1$, we get
\begin{equation*}
\sum_{\vecm} \lambda_{\pi_1\times \pi_2}(\vecm) \phi\left(\frac{N(\vecm)}{X}\right)
= \Res_{s=1} L(s, \pi_1\times \pi_2) \tilde{\phi}(1)X 
+ O_{\varepsilon}\big(C(\pi_1\times \pi_2)^{l(\sigma)+\varepsilon} X^\sigma\big), 
\end{equation*}
for any $\varepsilon>0$.
By \eqref{e:preconvex}, we get
\begin{equation*}
C(\pi_1\times\pi_2)^{l(\sigma)+\varepsilon}
\leq \big(C(\pi_1)C(\pi_2)\big)^{n(l(\sigma)+\varepsilon)}
= \big(C(\pi_1)C(\pi_2)\big)^{n(\frac{1-\sigma}{2}+\frac{1}{2}-\theta_n)+\varepsilon}.
\end{equation*}
Therefore, for $X\gg 1$ and $0< \sigma < 1$, for $\varepsilon>0$, we have 
\begin{multline}\label{e:RS_sum}
\sum_{\vecm} \lambda_{\pi_1\times\pi_2}(\vecm) \phi\left(\frac{N(\vecm)}{X}\right)
\\ = \Res_{s=1}L(s, \pi_1\times \pi_2) \tilde{\phi}(1) X + O_{\varepsilon, \phi, F, n} \left(X^\sigma \big((C(\pi_1)C(\pi_2)\big)^{n(\frac{1-\sigma}{2}+\frac{1}{2}-\theta_n)+\varepsilon}\right).
\end{multline}

For each non-archimedean place $v$, 
\begin{align*}
L(s, \pi_{1, v}\times \pi_{2, v}) 
& = \sum_{j=0}^\infty \frac{\lambda_{\pi_1\times \pi_2}(q_v^j)}{q_v^{js}}
\\ & = (1-q_v^{-ns})^{-1} \sum_{j_1=0}^\infty \cdots\sum_{j_{n-1}=0}^\infty \frac{\lambda_{\pi_1}(q_{v}^{j_1}, \ldots, q_{v}^{j_{n-1}}) \lambda_{\pi_2}(q_v^{j_1}, \ldots, q_v^{j_{n-1}})}
{q_v^{s\sum_{i=1}^{n-1}(n-i) j_i}}
\\ &= \sum_{j_0=0}^\infty \sum_{j_1=0}^\infty \cdots\sum_{j_{n-1}=0}^\infty \frac{\lambda_{\pi_1}(q_{v}^{j_1}, \ldots, q_{v}^{j_{n-1}}) \lambda_{\pi_2}(q_v^{j_1}, \ldots, q_v^{j_{n-1}})}
{q_v^{s\sum_{i=0}^{n-1}(n-i) j_i}}
\end{align*}
So, for each $j\geq 0$, 
\begin{equation*}
\lambda_{\pi_1\times \pi_2}(q_v^j) = \sum_{\substack{j_0, \ldots, j_{n-1}\geq 0\\ \sum_{i=0}^{n-1} (n-i)j_i = j}} \lambda_{\pi_1}(q_v^{j_1}, \ldots, q_v^{j_{n-1}}) \lambda_{\pi_2}(q_v^{j_1}, \ldots, q_v^{j_{n-1}}).
\end{equation*}
For example, when $n=2$, 
\begin{equation*}
\lambda_{\pi_1\times \pi_2}(q_v^j)
= \sum_{\substack{j_0, j_1 \geq 0\\ 2j_0 + j_1=j}}\lambda_{\pi_1}(q_v^{j_1}) \lambda_{\pi_2}(q_v^{j_2})
\end{equation*}
and when $n=3$, 
\begin{equation*}
\lambda_{\pi_1\times\pi_2}(q_v^j) = \sum_{\substack{j_0, j_1, j_2\geq 0\\ 3j_0+2j_1+j_2 = j}} \lambda_{\pi_1}(q_v^{j_1}, q_v^{j_2}) \lambda_{\pi_2}(q_v^{j_1}, q_v^{j_2}).
\end{equation*}

Recall from \cite{XiaLi10} that for $\pi \in \scrA_n(\leq Q)$
\begin{equation}\label{e:XiaLi_RSres}
 \Res_{s=1} L(s, \pi\times \widetilde{\pi}) = O_\varepsilon(Q^\varepsilon).
\end{equation}

\begin{theorem}\label{thm:spec_large_sieve_n}
For $X, Q\gg 1$ and $0< \sigma < 1$, we have 
\begin{multline}\label{e:spec_large_sieve_n}
\sum_{\pi\in \scrA_n(\leq Q)} \bigg|\sum_{\substack{\vecm_0, \ldots, \vecm_{n-1}\\ \prod_{i=0}^{n-1} N(\vecm_i)^{n-i} \leq X}} x_{\vecm_0, \ldots, \vecm_{n-1}} \lambda_\pi(\vecm_1, \ldots, \vecm_{n-1})\bigg|^2
\\ \ll_{\varepsilon, F, n} Q^\varepsilon\bigg(X+X^\sigma Q^{n(1-\sigma+1-2\theta_n)} \#\scrA_n(\leq Q) \bigg) 
\sum_{\substack{\vecm_0, \ldots, \vecm_{n-1}\\ \prod_{i=0}^{n-1} N(\vecm_i)^{n-i}\leq X}} |x_{\vecm_0, \ldots, \vecm_{n-1}}|^2, 
\end{multline} 
for any $\varepsilon>0$. 
\end{theorem}

\begin{proof}
By following the duality principle, in the proof of \cite[Theorem 4]{DK}, we note that, by general Hilbert theory, 
the inequality \eqref{e:spec_large_sieve_n} is equivalent to the following dual inequality:
\begin{multline}\label{e:spec_large_sieve_n_dual}
\sum_{\substack{\vecm_0, \ldots,\vecm_{n-1}\\ \prod_{i=0}^{n-1} N(\vecm_i)^{n-i}\leq X}} 
\bigg|\sum_{\pi\in \scrA_n(\leq Q)} x_\pi \lambda_\pi(\vecm_1, \ldots, \vecm_{n-1})\bigg|^2 
\\ \ll_{\varepsilon, F, n} Q^\varepsilon\bigg(X+X^\sigma Q^{n(1-\sigma+1-2\theta_n)} \#\scrA_n(\leq Q) \bigg) \sum_{\pi \in \scrA_n(\leq Q)}|x_\pi|^2, 
\end{multline}
for any $\varepsilon>0$. 
So we will prove \eqref{e:spec_large_sieve_n_dual}.

Choose a smooth and compactly supported function $\phi$ on $[0, +\infty]$ as in \pageref{ss:spectral_largesieve}.
Since each term of the LHS of \eqref{e:spec_large_sieve_n_dual} is positive and $\phi(x)=1$ when $x\in [0, 1]$ and $0\leq \phi(x)\leq 1$ for any $x\geq 1$, we have 
\begin{multline*}
\sum_{\substack{\vecm_0, \ldots,\vecm_{n-1}\\ \prod_{i=0}^{n-1} N(\vecm_i)^{n-i}\leq X}} 
\bigg|\sum_{\pi\in \scrA_n(\leq Q)} x_\pi \lambda_\pi(\vecm_1, \ldots, \vecm_{n-1})\bigg|^2 
\\ \leq \sum_{\vecm_0, \ldots,\vecm_{n-1}} \phi\left(\frac{\prod_{i=0}^{n-1}N(\vecm_i)^{n-i}}{X}\right)  
\bigg|\sum_{\pi\in \scrA_n(\leq Q)} x_\pi \lambda_\pi(\vecm_1, \ldots, \vecm_{n-1})\bigg|^2 
\\ = \sum_{\vecm_0, \ldots,\vecm_{n-1}} \phi\left(\frac{\prod_{i=0}^{n-1}N(\vecm_i)^{n-i}}{X}\right)  
\sum_{\pi_1, \pi_2\in \scrA_n(\leq Q)} x_{\pi_1}\overline{x_{\pi_2}} \lambda_{\pi_1}(\vecm_1, \ldots, \vecm_{n-1}) \lambda_{\widetilde{\pi_2}}(\vecm_1, \ldots, \vecm_{n-1}).
\end{multline*}
Here we use $\overline{\lambda_{\pi}(\vecm_1, \ldots, \vecm_{n-1})} = \lambda_{\tilde{\pi}}(\vecm_1, \ldots, \vecm_{n-1})$. 
After changing the order of the sums, by applying \cite[Lemma 1]{DK}, we have 
\begin{multline*}
= \sum_{\pi_1, \pi_2\in \scrA_n(\leq Q)} x_{\pi_1}\overline{x_{\pi_2}} 
\sum_{\vecm_0, \ldots, \vecm_{n-1}} \phi\left(\frac{\prod_{i=0}^{n-1} N(\vecm_i)^{n-i}}{X}\right) 
\lambda_{\pi_1}(\vecm_1, \ldots, \vecm_{n-1}) \lambda_{\widetilde{\pi_2}}(\vecm_1, \ldots, \vecm_{n-1})
\\ \leq M(\leq Q) \bigg(\sum_{\pi\in \scrA_n(\leq Q)} |x_{\pi}|^2\bigg), 
\end{multline*}
where
\begin{multline*}
M(\leq Q) 
\\ = \max_{\pi_1\in \scrA_n(\leq Q)} \bigg\{\sum_{\pi_2\in \scrA_n(\leq Q)} \sum_{\vecm_0, \ldots, \vecm_{n-1}} \lambda_{\pi_1}(\vecm_1, \ldots, \vecm_{n-1})\lambda_{\widetilde{\pi_2}}(\vecm_1, \ldots, \vecm_{n-1}) \phi\left(\frac{\prod_{i=0}^{n-1}N(\vecm_i)^{n-i}}{X}\right) \bigg\}.
\end{multline*}
Since 
\begin{multline*}
\sum_{\vecm_0, \ldots, \vecm_{n-1}} \lambda_{\pi_1}(\vecm_1, \ldots, \vecm_{n-1})\lambda_{\widetilde{\pi_2}}(\vecm_1, \ldots, \vecm_{n-1}) \phi\left(\frac{\prod_{i=0}^{n-1}N(\vecm_i)^{n-i}}{X}\right) 
\\ = \sum_{\vecm} \lambda_{\pi_1\times\widetilde{\pi_2}}(\vecm)\phi\left(\frac{N(\vecm)}{X}\right), 
\end{multline*}
by \eqref{e:RS_sum}, for $0<\sigma<1$ and any $\varepsilon>0$, we have
\begin{equation*}
= \Res_{s=1} L(s, \pi_1\times\widetilde{\pi_2}) \tilde{\phi}(1) X + O_{\varepsilon, \phi, F, n}\left(X^\sigma \big(C(\pi_1)C(\widetilde{\pi_2})\big)^{n(\frac{1-\sigma}{2}+\frac{1}{2}-\theta_n)+\varepsilon}\right). 
\end{equation*}
Note that $\Res_{s=1}L(s, \pi_1\times \widetilde{\pi_2})=0$ unless $\pi_2=\pi_1$. 
So 
\begin{multline*}
\sum_{\pi_2\in \scrA_n(\leq Q)} \sum_{\vecm_0, \ldots, \vecm_{n-1}} \lambda_{\pi_1}(\vecm_1, \ldots, \vecm_{n-1})\lambda_{\widetilde{\pi_2}}(\vecm_1, \ldots, \vecm_{n-1}) \phi\left(\frac{\prod_{i=0}^{n-1}N(\vecm_i)^{n-i}}{X}\right) 
\\ = \Res_{s=1} L(s, \pi_1\times \widetilde{\pi_1}) \tilde{\phi}(1) X 
+ O_{\varepsilon, \phi, F, n} \left(X^\sigma Q^{n(1-\sigma+1-2\theta_n)+\varepsilon}\#\scrA_n(\leq Q)\right)
\end{multline*}

Therefore, by applying \eqref{e:XiaLi_RSres}, we get 
\begin{multline*}
\sum_{\substack{\vecm_0, \ldots,\vecm_{n-1}\\ \prod_{i=0}^{n-1} N(\vecm_i)^{n-i}\leq X}} 
\bigg|\sum_{\pi\in \scrA_n(\leq Q)} x_\pi \lambda_\pi(\vecm_1, \ldots, \vecm_{n-1})\bigg|^2 
\\ \ll_{\varepsilon, \phi, F, n} Q^\varepsilon \bigg(  X+ X^\sigma Q^{n(1-\sigma+1-2\theta_n)} \#\scrA_n(\leq Q)\bigg)
\bigg(\sum_{\pi\in \scrA_n(\leq Q)} |x_{\pi}|^2\bigg)
\end{multline*}
and complete the proof of \eqref{e:spec_large_sieve_n_dual}.
By applying the general Hilbert theory as in the proof of \cite[Lemma 1]{DK}, this inequality is equivalent to the dual inequality \eqref{e:spec_large_sieve_n}. 
\end{proof} 

\subsection{Spectral large sieve for \texorpdfstring{$n=2$}{n=2} and symmetric square} 
Now we focus on the case when $n=2$. 
Fix $c>1$ and let $T>0$ be a large parameter and let $1<G<T$. 
Let $\scrA(T, G)$ be the set of irreducible cuspidal representations of $\PGL_2(\A)$ satisfying the followings:
$\pi=\otimes_v \pi_v$,  
\begin{itemize}
\item $\pi_v$ is unramified for all places $v$ of $F$;
\item $T\leq t_\pi = (\sum_{v\in S_\infty}\lambda_{\pi_v})^{\frac{1}{2}} \leq T+G$.
\end{itemize}
We assume that $T$ is sufficiently large so that \eqref{e:preconvexGL2} is valid. Note that we have \cite{zbMATH06321163}
\begin{equation}\label{weylauto}
\#\scrA(T, G) \ll_\epsilon T^{d-1+\epsilon}G
\end{equation}
where $d$ is the dimension of the symmetric space $\X_F$.

Since the central character is trivial, $\mu_1(\pi_v)=-\mu_2(\pi_v)$ for all $v\in S_\infty$.

We will use the argument that used to prove Theorem \ref{thm:spec_large_sieve_n} with a slight alternation to prove Theorem \ref{thm:spec_large_sieve_n=2_sym2} below. 
Since we consider $\scrA(T, G)$, not $\scrA_2(\leq Q)$, instead of 
$Q^{2(1-\sigma+1-2\theta_2)+\varepsilon}$ 
(for $n=2$), 
we first get the upper bound for the analytic conductors of $L(s, \pi_1\times\pi_2)$ for $\pi_1, \pi_2\in \scrA(T, G)$. 

By \eqref{e:C_R-S_def}, for $\pi_1, \pi_2\in \scrA(T, G)$, since $\theta_2=0$, we have
\begin{multline}\label{e:C_RS_n=2_est}
C(\pi_1\times \pi_2; t) 
\leq 
\prod_{v\in S_\infty} \bigg\{\left(1 + |t|+|\mu_1(\pi_{1, v})|+|\mu_1(\pi_{2, v})|\right)^2
\left(1 + |t|+\left||\mu_1(\pi_{1, v})|-|\mu_1(\pi_{2, v})|\right|\right)^2\bigg\}
\\ \leq 
\left(\left(1+|t|+G\right) 
\left(1+|t|+2T+2G\right)\right)^{2\#S_\infty}.
\end{multline}
Here We use that, since $\pi_1, \pi_2\in \scrA(T, G)$, 
we have $\left||\mu_1(\pi_{1, v})|-|\mu_1(\pi_{2, v})|\right| \leq G$ and $|\mu_1(\pi_{1, v})|+|\mu_1(\pi_{2, v})|\leq 2T+2G$. 

Now we get a formula which is similar to \eqref{e:RS_sum}.
Following the arguments in p.\pageref{e:RS_sum}, 
for $X\gg 1$, $0< \sigma < 1$ and any $\varepsilon>0$, we have 
\begin{multline*}
\sum_{\vecm}\lambda_{\pi_1\times\pi_2}(\vecm) \phi\left(\frac{N(\vecm)}{X}\right)
= \Res_{s=1}L(s, \pi_1\times\pi_2)\tilde\phi(1) X
\\ + O_{\varepsilon, \phi, F} \left(X^\sigma \left(q(\pi_1\times \pi_2) \left((1+G)(1+2T+2G)\right)^{2\#S_\infty}\right)^{\frac{-\sigma+1}{2}+\varepsilon}\right).
\end{multline*}
Here $\phi$ satisfies the conditions in the proof of Theorem \ref{thm:spec_large_sieve_n}. 
Since 
\begin{equation*}
(1+G)(1+2T+2G)\ll TG, 
\end{equation*}
we have 
\begin{multline}\label{e:RS_sum_n=2}
\sum_{\vecm}\lambda_{\pi_1\times\pi_2}(\vecm) \phi\left(\frac{N(\vecm)}{X}\right)
= \Res_{s=1}L(s, \pi_1\times\pi_2)\tilde\phi(1) X
+ O_{\varepsilon, \phi, F} \left(X^\sigma (TG)^{\#S_\infty(1-\sigma)+\varepsilon}\right).
\end{multline}

Let $\pi=\otimes_v \pi_v$ be an irreducible cuspidal representation of $\PGL_2(\A)$.
Assume that the central character for $\pi$ is trivial and $\pi_v$ is unramified for any $v$. 
We define $\pi^{(2)} = \sym^2\pi = \otimes_{v} \pi^{(2)}_v$, the symmetric square representation of $\pi$ which is an irreducible cuspidal representation of $\PGL_3(\A)$, 
with trivial central character and $\pi^{(2)}_v$ is unramified for any $v$.
The Langlands parameters and Satake parameters for $\pi^{(2)}_v$ are determined by 
\begin{equation}\label{e:mu_sym2}
\mu_1(\pi_v^{(2)}) = 2\mu_1(\pi_v), \quad \mu_2(\pi^{(2)}_v)=2\mu_2(\pi_v)  \quad\text{ and } \quad \mu_3(\pi^{(2)}_v)=0, 
\end{equation}
when $v$ is archimedean and 
\begin{equation*}
\alpha_1(\pi^{(2)}_v) = \alpha_1(\pi_v)^2, \quad \alpha_2(\pi^{(2)}_v) = \alpha_2(\pi_v)^2 \quad\text{ and } \quad \alpha_3(\pi^{(2)}_v) = 1
\end{equation*}
when $v$ is non-archimedean. 
The $L$-function for $\sym^2\pi = \pi^{(2)}$ is defined by 
\begin{equation*}
L(s, \sym^2\pi) = \prod_{v<\infty} (1-q_v^{-s})^{-1} (1-\alpha_1(\pi_v)^2 q_v^{-s})^{-1} (1-\alpha_2(\pi_v)^2q_v^{-s})^{-1}
= \sum_{\vecm}\frac{\lambda_{\sym^2\pi}(\vecm)}{N(\vecm)^s}
\end{equation*}
and the completed $L$-function 
\begin{equation*}
\Lambda(s, \sym^2\pi) = L(s, \sym^2\pi) \prod_{v\in S_\infty} \Gamma_{F_v}(s) \Gamma_{F_v}(s+2\mu_1(\pi_v)) \Gamma(s+2\mu_2(\pi_v)).
\end{equation*}
Note that the Euler product and series converges absolutely for $\Re(s)>1$. 
Then $\Lambda(s, \sym^2\pi)$ continues analytically to $s\in \C$ and satisfies the $\PGL_3$-functional equation. 

For $\pi_1, \pi_2\in \scrA(T, G)$, recalling \eqref{e:C_R-S_def} and \eqref{e:mu_sym2}, and from the assumption that $T$ is sufficiently large so that $\pi_1$ and $\pi_2$ are tempered \cite{CLLL}, we have
\begin{multline*}
C(\pi_1^{(2)}\times \pi_2^{(2)}; t) 
\leq q(\pi_1^{(2)} \times \pi_2^{(2)}) 
\prod_{v\in S_\infty} \bigg\{(1+|t|)
\prod_{j=1}^2\left(1+|it+2\mu_{j}(\pi_{1, v})|\right)
\prod_{j=1}^2\left(1+|it+2\mu_{j}(\pi_{2, v})|\right)
\\ \times 
\prod_{i_1=1}^2\prod_{i_2=1}^2 \left(1+ 
|it+2\mu_{i_1}(\pi_{1, v})+2\mu_{i_2}(\pi_{2, v})|\right)
\bigg\}.
\end{multline*}
As in \eqref{e:C_RS_n=2_est}, we have
\begin{equation*}
C(\pi_1^{(2)}\times \pi_2^{(2)}; t) 
\ll_{t, F} \left((T+G)^4 (T+G)^2G^2\right)^{\#S_\infty}
\ll \left(T^3G\right)^{2\#S_\infty}. 
\end{equation*}
Then, for $X\gg 1$ and $0< \sigma< 1$, for any $\epsilon>0$, we have 
\begin{multline}\label{e:RS_sum_sym2}
\sum_{\vecm}\lambda_{\pi^{(2)}_1\times \pi^{(2)}_2}(\vecm) \phi\left(\frac{N(\vecm)}{X}\right)
= \Res_{s=1} L(s, \pi^{(2)}_1\times \pi^{(2)}_2) \tilde{\phi}(1) X 
+ O_{\varepsilon, \phi, F} \big(X^\sigma (T^3G)^{\#S_\infty(1-\sigma)+\varepsilon} \big),
\end{multline}
where we used \eqref{e:preconvexGL2}.
 
Now recall from \cite{XiaLi10} that
\begin{equation*}
\Res_{s=1}L(s, \pi\times \tilde{\pi}) \ll_\varepsilon T^\varepsilon
\end{equation*}
and that
\begin{equation*}
\Res_{s=1}L(s, \pi^{(2)}\times \tilde{\pi}^{(2)})\ll_\varepsilon T^\varepsilon. 
\end{equation*}
\begin{theorem}\label{thm:spec_large_sieve_n=2_sym2}
For $X\gg 1$ and $0< \sigma < 1$, we have 
\begin{multline}\label{e:spec_large_sievei_n=2}
\sum_{\pi\in \scrA(T, G)} \bigg|\sum_{\substack{\vecm_0, \vecm_1\\ N(\vecm_0)^2 N(\vecm_1) \leq X}} x_{\vecm_0, \vecm_1} \lambda_\pi(\vecm_1)\bigg|^2
\\ \ll_{\varepsilon, F} T^\varepsilon\left( X + X^\sigma (TG)^{\#S_\infty (1-\sigma)} \#\scrA(T, G)\right)
\sum_{\substack{\vecm_0, \vecm_1\\ N(\vecm_0)^2N(\vecm_1)\leq X}} |x_{\vecm_0, \vecm_1}|^2
\end{multline}
and 
\begin{multline}\label{e:spec_large_sievei_n=2_sym}
\sum_{\pi\in \scrA(T, G)} \bigg|\sum_{\substack{\vecm_0, \vecm_1, \vecm_2\\ N(\vecm_0)^3N(\vecm_1)^2N(\vecm_2)\leq X}} 
x_{\vecm_0, \vecm_1, \vecm_2}\lambda_{\pi^{(2)}}(\vecm_1, \vecm_2)\bigg|^2
\\ \ll_{\varepsilon, F} T^\varepsilon\left( X+ X^{\sigma} (T^3G)^{\#S_\infty(1-\sigma)} \#\scrA(T, G)\right) 
\sum_{\substack{\vecm_0, \vecm_1, \vecm_2\\ N(\vecm_0)^3N(\vecm_1)^2N(\vecm_2)\leq X}} |x_{\vecm_0, \vecm_1, \vecm_2}|^2. 
\end{multline}
\end{theorem}

\begin{proof}
We follow the proof of Theorem \ref{thm:spec_large_sieve_n}. 
We again choose a smooth and compactly supported function $\phi$ on $[0, +\infty]$ satsifying conditions in p.\pageref{ss:spectral_largesieve}. 
Consider 
\begin{multline*}
\sum_{\substack{\vecm_0, \vecm_1\\ N(\vecm_0)^2 N(\vecm_1) \leq X}} \bigg|\sum_{\pi\in \scrA(T, G)} x_\pi \lambda_\pi(\vecm_1)\bigg|^2
\leq \sum_{\vecm_0, \vecm_1} \phi\left(\frac{N(\vecm_0)^2 N(\vecm_1)}{X}\right)
\bigg|\sum_{\pi\in \scrA(T, G)} x_\pi \lambda_\pi(\vecm_1)\bigg|^2
\\ \leq \max_{\pi_1\in \scrA(T, G)} \bigg\{\sum_{\pi_2\in \scrA(T, G)}\sum_{\vecm}\lambda_{\pi_1\times\widetilde{\pi_2}}(\vecm) \phi\left(\frac{N(\vecm)}{X}\right)\bigg\} 
\sum_{\pi\in \scrA(T, G)} |x_\pi|^2. 
\end{multline*}
By \eqref{e:RS_sum_n=2}, for $0< \sigma < 1$, we have 
\begin{multline*}
\sum_{\pi_2\in \scrA(T, G)}\sum_{\vecm}\lambda_{\pi_1\times\widetilde{\pi_2}}(\vecm) \phi\left(\frac{N(\vecm)}{X}\right)
= \Res_{s=1} L(s, \pi_1\times \tilde{\pi_1}) \tilde{\phi}(1) X
\\ + O_{\varepsilon, \phi, F} \left(X^\sigma (TG)^{\#S_\infty (1-\sigma)+\varepsilon} \#\scrA(T, G)\right)
\end{multline*}
Therefore we get
\begin{multline*}
\sum_{\substack{\vecm_0, \vecm_1\\ N(\vecm_0)^2 N(\vecm_1) \leq X}} \bigg|\sum_{\pi\in \scrA(T, G)} x_\pi \lambda_\pi(\vecm_1)\bigg|^2
\\ \ll_{\varepsilon, F} T^\varepsilon\bigg( X + X^\sigma (TG)^{\#S_\infty (1-\sigma)} \#\scrA(T, G)\bigg)
\sum_{\pi\in \scrA(T, G)} |x_\pi|^2.
\end{multline*}
By the duality, we get \eqref{e:spec_large_sievei_n=2}.

Similarly, we consider the symmetric square of $\pi\in \scrA(T, G)$: 
\begin{multline*}
\sum_{\substack{\vecm_0, \vecm_1, \vecm_2\\ N(\vecm_0)^3N(\vecm_1)^2N(\vecm_2)\leq X}}
\bigg|\sum_{\pi\in \scrA(G, T)} x_{\pi}\lambda_{\pi^{(2)}}(\vecm_1, \vecm_2)\bigg|^2
\\ \leq \sum_{\vecm_0, \vecm_1, \vecm_2} \phi\left(\frac{N(\vecm_0)^3N(\vecm_1)^2N(\vecm_2)}{X}\right) 
\bigg|\sum_{\pi\in \scrA(T, B)} x_\pi\lambda_{\pi^{(2)}}(\vecm_1, \vecm_2)\bigg|^2
\\ \leq \max_{\pi_1\in \scrA(T, G)} \bigg\{\sum_{\pi_2\in \scrA(T, G)}
\sum_{\vecm} \lambda_{\pi^{(2)}_1\times \widetilde{\pi_2}^{(2)}} (\vecm) \phi\left(\frac{N(\vecm)}{X}\right)\bigg\} 
\sum_{\pi\in \scrA(T, G)} |x_\pi|^2. 
\end{multline*}
By \eqref{e:RS_sum_sym2}, for $0< \sigma<1$, for any $\varepsilon>0$, we have 
\begin{multline*}
\sum_{\pi_2\in \scrA(T, G)}
\sum_{\vecm} \lambda_{\pi^{(2)}_1\times \widetilde{\pi_2}^{(2)}} (\vecm) \phi\left(\frac{N(\vecm)}{X}\right)
= \Res_{s=1} L(s, \pi_1^{(2)} \times \widetilde{\pi_1}^{(2)}) \tilde{\phi}(1) X
\\ + O_{\varepsilon, \phi, F}\big(X^{\sigma} (T^3G)^{\#S_\infty(1-\sigma)+\varepsilon}\big).
\end{multline*}
Therefore we get
\begin{multline*}
\sum_{\substack{\vecm_0, \vecm_1, \vecm_2\\ N(\vecm_0)^3N(\vecm_1)^2N(\vecm_2)\leq X}}
\bigg|\sum_{\pi\in \scrA(G, T)} x_{\pi}\lambda_{\pi^{(2)}}(\vecm_1, \vecm_2)\bigg|^2
\\ \ll_{\varepsilon, F} T^\varepsilon\left( X+ X^{\sigma} (T^3G)^{\#S_\infty(1-\sigma)}\#\scrA(T, G)\right) 
\sum_{\pi\in \scrA(T, G)} |x_\pi|^2. 
\end{multline*}
By the duality, we obtain \eqref{e:spec_large_sievei_n=2_sym}.
\end{proof} 

\subsection{Bounding the dimension of a joint eigenspace}
We now follow the proof of \cite[Theorem 3]{DK}.
Let $\pi=\otimes_v\pi_v$ be an irreducible cuspidal representation of $\PGL_2(\A)$ with the trivial central character. 
Assume that $\pi_v$ is unramified for any place $v$.

We fix $\pi_0=\otimes_{v} \pi_{0, v}\in \scrA(T, G)$. 
For a fixed $\alpha>0$, let 
\begin{equation*}
\scrP(T, \alpha) = \left\{v\in S_{\finite} :\, q_v \leq (\log T)^\alpha\right\}
\end{equation*}
and let 
\begin{equation*}
\scrA_{\pi_0} (T, G; \alpha) = \left\{\pi=\otimes_v \pi_v \in \scrA(T, G):\, \pi_v\cong \pi_{0, v} \text{ for any } v\in \scrP(T, \alpha)\right\}. 
\end{equation*}
Then whenever $\pi\in\scrA_{\pi_0}(T, G; \alpha)$, we get 
\begin{equation*}
\lambda_{\pi}(q_v) = \lambda_{\pi_0}(q_v)\quad \text{ for all } v<\infty \text{ with } q_v \leq (\log T)^{\alpha}.
\end{equation*}
So we have 
\begin{equation*}
\lambda_{\pi_0}(\vecm) = \lambda_{\pi}(\vecm) \quad \text{ when } N(\vecm)\leq X.
\end{equation*}

Let $\pi_1=\otimes_v \pi_{1, v}$ and $\pi_2=\otimes_{v} \pi_{2, v}$ are irreducible cuspidal representations of $\PGL_2(\A)$ with the trivial central character, and we assume that $\pi_{1, v}$ and $\pi_{2, v}$ are unramified for all places $v$ of $F$.
By \cite[Theorem 4.1.2]{R00}, if $\sym^2\pi_{1, v}\cong \sym^2 \pi_{2, v}$ for almost all places $v$ of $F$, 
Then $\pi_1\cong \pi_2$.

Our goal is to obtain the upper bound for $\#\scrA_{\pi_0}(T, G;\alpha)$. 
To obtain this, we need to get a non-trivial lower bound for $\sum_{\substack{\vecm\\ N(\vecm)\leq X}} |\lambda_{\pi}(\vecm)|^2$ for a given irreducible cuspidal representation $\pi=\otimes_v\pi_v$ of $\GL_2(\A)$ with the trivial central character and $\pi_v$ is unramified for any $v$. 

For each $v\in S_{\finite}$, since $\lambda_{\pi}(q_v^2) = \alpha_1(\pi_v)^2+\alpha_2(\pi_v)^2 = \lambda_{\pi}(q_v)^2-1$, 
if $|\lambda_{\pi}(q_v)| <\frac{1}{2}$ then $|\lambda_{\pi}(q_v^2)|\geq \big|1-|\lambda_{\pi}(q_v)|^2\big| >\frac{3}{4}>\frac{1}{2}$. 
So either $|\lambda_{\pi}(q_v)|\geq \frac{1}{2}$ or $|\lambda_{\pi}(q_v^2)|>\frac{1}{2}$. 
For $i\in \{1, 2\}$, let 
\begin{equation*}
\scrP_i(T, \alpha) = \left\{v\in S_{\finite}:\, q_v\leq (\log T)^\alpha \text{ and } |\lambda_{\pi_0}(q_v^i)|\geq \frac{1}{2}\right\}.
\end{equation*}
Then 
\begin{equation*}
\scrP_1(T, \alpha)\cup \scrP_2(T, \alpha)
= \scrP(T, \alpha).
\end{equation*}
Moreover, for at least one of $\{1, 2\}$ we have 
\begin{equation*}
\#\scrP_i(T, \alpha) \geq \frac{1}{2}\#\scrP(T, \alpha). 
\end{equation*}
Fix $i\in \{1, 2\}$. 

By  Landau's prime ideal theorem, 
\begin{equation*}
\#\scrP(T, \alpha) \sim \frac{(\log T)^\alpha}{\alpha\log\log T}. 
\end{equation*}

Then there exists a constant $c>0$ such that 
$\#\scrP(T, \alpha) \geq c \frac{(\log T)^\alpha}{\alpha\log\log T}$
and 
\begin{equation*}
\#\scrP_i(T, \alpha) \geq \frac{1}{2}\scrP(T, \alpha) \geq \frac{c}{2} \frac{(\log T)^\alpha}{\alpha\log\log T}.
\end{equation*}
We choose $\ell$ (will be determined later) and let
\begin{equation*}
\scrN_\ell^i(T, \alpha) = \left\{\bp_{v_1}\cdots \bp_{v_\ell}:\, v_1, \ldots, v_\ell \in \scrP_i(T, 
\alpha), \text{ distinct} \right\}
\end{equation*}
For $\vecm\in \scrN_\ell^i(T, \alpha)$, we have $|\lambda_{\pi_0}(\vecm^i)|\geq 2^{-\ell}$.
We have 
\begin{equation*}
\max \scrN_\ell^i(T, \alpha) \leq (\log T)^{\alpha\ell}.
\end{equation*}
We choose $\ell$ such that $(\log T)^{\alpha\ell} < T^\beta$ but $(\log T)^{\alpha\ell}$ near $T^\beta$.
So 
\begin{equation*}
\max \scrN_\ell^i(T, \alpha) < T^\beta.
\end{equation*}
For $N(\vecm)\leq T^\beta$, we have 
\begin{equation*}
x_{\vecm} = \begin{cases}
\overline{\lambda_{\pi_0}(\vecm)^i} = \lambda_{\widetilde{\pi_0}}(\vecm^i) & \text{ for }\vecm\in \scrN_\ell^i(T, \alpha), \\ 0 & \text{ otherwise. }
\end{cases}
\end{equation*}
When $i=1$, for $\vecm_0$, $\vecm_1$, let 
\begin{equation*}
x_{\vecm_0, \vecm_1} = x_{\vecm_0^2 \vecm_1}. 
\end{equation*}
When $i=2$, for $\vecm_0$, $\vecm_1$, $\vecm_2$, let 
\begin{equation*}
x_{\vecm_0, \vecm_1, \vecm_2} = x_{\vecm_0^3 \vecm_1^2\vecm_2}. 
\end{equation*}
Then we apply Theorem \ref{thm:spec_large_sieve_n=2_sym2}.
Since $x_\vecm$ is supported only on $\vecm\in \scrN_\ell^i(T, \alpha)$, 
when $i=1$, $x_{\vecm_0, \vecm_1}=x_{\vecm_0^2\vecm_1}=0$ if $\vecm_0\notin \bo_F^\times$. 
Similarly when $i=2$, $x_{\vecm_0, \vecm_1, \vecm_2}=x_{\vecm_0^3\vecm_1^2\vecm_2}=0$ if $\vecm_0$ or $\vecm_1$ is not in $\bo_F^\times$. 
Taking $X=T^\beta$, by Theorem \ref{thm:spec_large_sieve_n=2_sym2}, 
\begin{multline*}
\#\scrA_{\pi_0}(T, G, \alpha) \bigg|\sum_{\vecm\in \scrN_\ell^i(T, \alpha)} |\lambda_{\pi_0}(\vecm^i)|^2\bigg|^2
\leq \sum_{\pi\in \scrA(T, G)} \bigg|\sum_{\substack{\vecm\\ N(\vecm)\leq (\log T)^{\alpha\ell}}} x_{\vecm} \lambda_{\pi}(\vecm^i)\bigg|^2
\\ \ll_{\varepsilon, F} T^\varepsilon\big(T^{\beta} + T^{\beta\sigma} (T^{2i-1}G)^{\#S_\infty (1-\sigma)} \#\scrA(T, G)\big) 
\sum_{\vecm\in \scrN_\ell^i(T, \alpha)} |x_{\vecm}|^2.
\end{multline*}
From \eqref{weylauto}, we get
\begin{equation*}
\ll_{\varepsilon, F} T^\varepsilon\big( T^\beta + T^{\beta\sigma+d-1}G (T^{2i-1}G)^{\# S_\infty (1-\sigma)}\big)
\sum_{\vecm\in \scrN_\ell^i(T, \alpha)} |\lambda_{\pi_0}(\vecm^i)|^2
\end{equation*}
 For $\sum_{\vecm\in \scrN_\ell(T, \alpha)} |\lambda_{\pi_0}(\vecm^i)|^2 \geq \#\scrN_\ell(T, \alpha) 2^{-2\ell}$, 
we get
\begin{equation*}
\#\scrA_{\pi_0}(T, G, \alpha)
\ll_{\varepsilon, F}
T^\varepsilon\big( T^{\beta} + T^{\beta\sigma+d} (T^3G)^{\#S_\infty (1-\sigma)}\big)
\#\scrN_\ell^i(T, \alpha)^{-1} 2^{2\ell},
\end{equation*}
where we used $T^{2i-1}\leq T^3$. We choose $\ell$ and $\beta$ such that $(\log T)^{\alpha\ell} < T^\beta$. 
We can take 
\begin{equation}\label{e:ell}
\ell = \left\lfloor \frac{\beta\log T}{\alpha\log \log T}\right\rfloor.
\end{equation}
 By the fundamental theorem of ideal theory in number fields, 
\begin{equation*}
\#\scrN_\ell^i(T, \alpha) \geq \frac{(\#\scrP_i(T, \alpha))!}{\ell! (\#\scrP_i(T, \alpha)-\ell)!}.
\end{equation*}
Recall Stirling's approximation, for any positive integer $m$, 
\begin{equation*}
\sqrt{2\pi} m^{m+\frac{1}{2}} e^{-m} \leq m! \leq em^{m+\frac{1}{2}} e^{-m}.
\end{equation*}
Let $P=\#\scrP_i(T, \alpha)$. 
For $P\geq \frac{c}{2} \frac{(\log T)^{\alpha}}{\alpha\log \log T}$ 
and $\ell$ as in \eqref{e:ell}, 
\begin{align*}
\frac{P!}{\ell!(P-\ell)!}
& \geq \frac{\sqrt{2\pi}}{e^2}
\frac{P^{P+\frac{1}{2}}e^{-P}}
{\ell^{\ell+\frac{1}{2}} e^{-\ell} (P-\ell)^{P-\ell+\frac{1}{2}} e^{-P+\ell}}
\\ & = \frac{\sqrt{2\pi}}{e^2} \left(\frac{P}{P-\ell}\right)^{P-\ell+\frac{1}{2}} \left(\frac{P}{\ell}\right)^{\ell}\ell^{-\frac{1}{2}}
\\ & \gg \left(\frac{P}{\ell}\right)^{\ell} \ell^{-\frac{1}{2}}.
\end{align*}
Since 
\begin{equation*}
\frac{P}{\ell}\geq \frac{c}{2} \frac{(\log T)^{\alpha-1}}{\beta}, 
\end{equation*}
we have 
\begin{equation*}
\frac{P!}{\ell!(P-\ell)!} \gg (c2^{-1})^{\ell} \frac{(\log T)^{(\alpha-1)\ell}}{\beta^\ell}\ell^{-\frac{1}{2}}.
\end{equation*}
For 
\begin{equation*} 
(\log T)^{\ell} = (\log T)^{\frac{\beta\log T}{\alpha\log\log T}-a} \gg T^{\frac{\beta}{\alpha}}
\end{equation*}
where $a=\frac{\beta\log T}{\alpha\log \log T}-\ell$ and $0\leq a < 1$, 
we have 
\begin{equation*}
\frac{P!}{\ell!(P-\ell)!} \gg (c2^{-1}\beta^{-1})^\ell \ell^{-\frac{1}{2}} T^{\frac{(\alpha-1)\beta}{\alpha}}.
\end{equation*}

Therefore we get
\begin{equation*}
\#\scrA_{\pi_0} (T, G, \alpha) 
\ll_{\varepsilon, F} 
\big(R^{(i)}(T, G) T^{\beta} + T^{\beta\sigma+d-1}G (T^3G)^{\#S_\infty (1-\sigma)+\varepsilon}\big)
(c^{-1}8\beta)^{\ell} \ell^{\frac{1}{2}} T^{-\frac{(\alpha-1)\beta}{\alpha}}.
\end{equation*}
For sufficiently large $T$, for a constant $c_0$, we have 
\begin{equation*}
(c_0\beta)^{\ell}= e^{\ell\log (c_0\beta)} 
\leq \exp\left(\frac{\beta\log T}{\alpha\log\log T}\log (c_0\beta)\right)
= T^{\frac{\beta\log(c_0\beta)}{\alpha}\frac{1}{\log\log T}} \ll T^\varepsilon. 
\end{equation*}
Moreover $\ell^{\frac{1}{2}}\ll T^{\varepsilon}$ for sufficiently large $T$.
Then we get
\[
\#\scrA_{\pi_0} (T, G, \alpha) 
\ll_{\varepsilon, F} 
T^\varepsilon\big(T^\beta + T^{\beta\sigma+d-1}G (T^3G)^{\#S_\infty (1-\sigma)}\big)
T^{-\frac{(\alpha-1)\beta}{\alpha}+\varepsilon},
\]
and we take $G=1$, $\sigma=0$, and $\beta =d-1+ 3\#S_\infty$ to conclude that
\begin{equation*}
\#\scrA_{\pi_0}(T, 1, \alpha) \ll_{\varepsilon, F} 
T^{\frac{d-1+3\#S_\infty}{\alpha}+\varepsilon}.
\end{equation*}

\section{Effective multiplicity one}

In this section, we again assume that $\G=\GL_2$ and $\bar{\G} =\PGL_2$. 
Let $f$ be a Maass cusp form which is $L^2$-normalized. Define 
\begin{equation*}
\rho(f) = \int_{F_\infty^\times} \left|W_{f, \infty}(y)\right|^2 \, d_\infty^\times y.
\end{equation*}
Recall that $f(g)$ has the following Fourier expansion (see \S\ref{ss:notation}):
for $x\in \A$, $y=(y_\infty, y_{\finite})\in \A^\times$ and $\kappa\in \maxK$, 
\begin{equation}\label{e:f_Fourier}
f(n(x)a(y)\kappa)
= \sum_{\gamma\in \bo_F^\times\backslash (y)^{-1}\bd^{-1}\cap F^\times} \rho_f(y, \gamma) 
\sum_{\eta\in \bo_F^\times} W_{f_\infty}(a(\eta\gamma y)) \varphi(\gamma x)
\end{equation}
By \cite[(2.10)]{BK11}, $W_{f_\infty}(g) = \prod_{v\in S_\infty} W_{f, v}(g_v)$ where 
\begin{equation}\label{e:Wfv_archimedean}
W_{f, v}(y) = \begin{cases}
\sqrt{|y|} K_{it_{f, v}}(2\pi |y|) & \text{ when } F_v=\R, \\
|y| K_{2it_{f, v}}(4\pi |y|) & \text{ when } F_v=\C.
\end{cases}
\end{equation}
For any non-zero integral ideal $\ba \subset F$, there exists a uniquely determined $t_j\in \A_{\finite}^\times$ such that $\ba= (\gamma)(t_j) \bd$ for some $\gamma\in F^\times$.
We define
\begin{equation}\label{e:rhof_a_def}
\rho_f(\ba) = \rho_f(t_j, \gamma) \sqrt{N(\ba)}.
\end{equation}
Note that $N(\ba) = N(\gamma) \|t_j\|^{-1} N(\bd)$.

Recall that the Eisenstein series given by
\begin{equation*}
E(g, s) = \sum_{\gamma\in \BorelB(F) \backslash\G(\A)} H(\gamma g)^s
\end{equation*} 
converges absolutely for  $\Re(s)> 1$ and it continues to an analytic function for $s\in \C$ except a simple pole at $s=1$ \cite{MR379375}. 
Let
\begin{equation*}
c_F = \Res_{s=1} E(g, s).
\end{equation*}

Following \cite[Chapter V]{Jac72}, we prove the following lemma. 
\begin{lemma}\label{e:inner_f^2E}
For $\Re(s)>1$, we have 
\begin{equation*}
N(\bd)^{-s+1} \int_{\bar{\G}(F) \backslash \bar{\G}(\A)} |f(g)|^2 E(g, s)\, dV(g)
= G(s; f) \sum_{\ba} \frac{|\rho_f(\ba)|^2}{N(\ba)^{s}}. 
\end{equation*}
Here
\begin{multline}\label{e:Gsf_def}
G(s; f)=\prod_{v\in S_\infty} \bigg\{2^{\epsilon_v} 
\frac{(\pi(1+\epsilon_v))^{-(1+\epsilon_v)s}}{8} 
\\ \times \frac{\Gamma\left(\frac{(1+\epsilon_v)s}{2}\right) \Gamma\left(\frac{(1+\epsilon_v)s}{2}\right) \Gamma\left(\frac{(1+\epsilon_v)s}{2}+i(1+\epsilon_v)t_{f, v}\right) \Gamma\left(\frac{(1+\epsilon_v)s}{2}-i(1+\epsilon_v)t_{f, v}\right)}
{\Gamma((1+\epsilon_v)s)}\bigg\}
\end{multline}

Let 
\begin{equation}\label{e:rhof}
\rho(f) = \prod_{v\in S_\infty} \rho(f_v)
= \prod_{F_v=\R} \frac{\pi}{8\cosh(\pi t_{f, v})} \prod_{F_v=\C} \frac{t_{f, v}}{8\pi \sinh(2\pi t_{f, v})}. 
\end{equation}
We have 
\begin{equation}\label{e:rhofrhof(ba)_size}
|\rho(f)|^{\frac{1}{2}} |\rho_f(\ba)| = O_\varepsilon(N(\ba)^{\theta+\varepsilon}), 
\end{equation}
for any $\varepsilon>0$. 
Here, $\theta$ is the best constant towards the generalized Ramanujan conjecture. The best known bound is $\theta\geq \frac{7}{64}$ \cite{MR2811610}.
\end{lemma}

\begin{proof}
We compute the integral by unfolding the series $E(g, s)$ with $|f(g)|^2$:
\begin{align*}
\int_{\bar{\G}(F)\backslash \bar{\G}(\A)} |f(g)|^2 E(g, s)\, dV(g)
& = \int_{\BorelB(F)\CentZ(\A) \N(\A) \backslash \G(\A)} |f(g)|^2 \|H(g)\|^s\, dV(g)
\\ & = \int_{F^\times\backslash \A^\times} 
\bigg(\int_{F\backslash \A} |f(n(x)a(y))|^2 \, dx\bigg)  \|y\|^{s-1}\, d^\times y
\\ & = \int_{F^\times \backslash \A^\times \prod_{v<\infty} \bo_v^\times}
\bigg(\int_{F\backslash \A} |f(n(x)a(y))|^2 \, dx\bigg) 
\|y\|^{s-1}\, d^\times y
\\ & = \int_{\bo_F^\times\backslash F_\infty^\times} \sum_{j=1}^h 
\bigg(\int_{F\backslash \A} |f(n(x)a(y, t_j)|^2 \, dx\bigg) 
\|y\|_\infty^{s-1}\|t_j\|^{s-1} \, d_\infty^\times y.
\end{align*}
Applying \eqref{e:f_Fourier}, 
\begin{multline*}
= \sum_{j=1}^h 
\sum_{\gamma\in \bo_F^\times \backslash(t_j)^{-1} \bd^{-1}\cap F^\times} 
|\rho_f(t_j, \gamma)|^2
\int_{\bo_F^\times\backslash F_\infty^\times} 
\sum_{\eta\in \bo_F^\times}
|W_{f_\infty}(a(\gamma \eta y))|^2
\|y\|_\infty^{s-1}\|t_j\|^{s-1} \, d_\infty^\times y.
\end{multline*}
Unfolding, we have 
\begin{multline*}
= \sum_{j=1}^h 
\sum_{\gamma\in \bo_F^\times \backslash(t_j)^{-1} \bd^{-1}\cap F^\times} 
|\rho_f(t_j, \gamma)|^2
\int_{F_\infty^\times} 
|W_{f_\infty}(a(\gamma y))|^2
\|y\|_\infty^{s-1}\|t_j\|^{s-1} \, d_\infty^\times y
\\ = \sum_{j=1}^h 
\sum_{\gamma\in \bo_F^\times \backslash(t_j)^{-1} \bd^{-1}\cap F^\times} 
|\rho_f(t_j, \gamma)|^2 \|\gamma\|_\infty^{-s+1}\|t_j\|^{s-1}
\int_{F_\infty^\times} 
|W_{f_\infty}(a(y))|^2
\|y\|_\infty^{s-1} \, d_\infty^\times y.
\end{multline*}
For any non-zero, integral ideal $\ba\subset F$, there exists uniquely determined $t_j\in \A_{\finite}^\times$ for $1\leq j\leq h$ such that $\ba = (\gamma)(t_j)\bd$ for some $\gamma\in F$.
By \eqref{e:rhof_a_def}, we get
\begin{equation*}
\int_{\bar{\G}(F) \backslash \bar{\G}(\A)} |f(g)|^2 E(g, s)\, dV(g)
= N(\bd)^{s-1}\sum_{\ba} \frac{|\rho_f(\ba)|^2}{N(\ba)^{s}}
\int_{F_\infty^\times} |W_{f_\infty}(a(y))|^2 \|y\|_\infty^{s-1} \, d_\infty^\times y.
\end{equation*}
Here we use that $N(\ba) = N(\gamma) \|t_j\|^{-1} N(\bd) =\|\gamma\|_\infty \|t_j\|^{-1} N(\bd)$, since $\|t_j\| = \prod_{v<\infty} |t_{j, v}|_v$. 

We now compute the integral: applying \eqref{e:Wfv_archimedean}, 
\begin{equation*}
\int_{F_\infty^\times} |W_{f_\infty}(a(y))|^2 \|y\|_\infty^{s-1} \, d_\infty^\times y
= \prod_{v\in S_\infty}
\int_{F_v^\times} |y|^{1+\epsilon_v} K_{i(1+\epsilon_v)t_{f, v}}(2\pi(1+\epsilon_v)|y|)^2 |y|^{(1+\epsilon_v)(s-1)} \, d_v^\times y.
\end{equation*}
When $F_v=\R$, 
\begin{equation*}
\int_{\R^\times} K_{it_{f, v}}(2\pi|y|)^2 |y|^{s}\, \frac{dy}{2|y|}
= \int_0^\infty K_{it_{f, v}}(2\pi y)^2 y^{s}\, \frac{dy}{y}.
\end{equation*}
When $F_v=\C$, 
\begin{equation*}
\int_{\C} K_{i2t_{f, v}}(4\pi |y|)^2 |y|^{2s}\, \frac{dy}{\pi|y|^2}
= \int_{0}^{2\pi}\int_0^\infty K_{i2t_{f, v}}(4\pi y)^2 y^{2s} 
\, \frac{ydy\, d\theta}{\pi y^2}
= 2\int_0^\infty K_{i2t_{f, v}}(4\pi y)^2 y^{2s}\, \frac{dy}{y}. 
\end{equation*}
Combining, we can write 
\begin{equation*}
\int_{F_\infty^\times} |W_{f_\infty}(a(y))|^2 \|y\|_\infty^{s-1} \, d_\infty^\times y
= \prod_{v\in S_\infty} 2^{\epsilon_v} \int_0^\infty K_{i(1+\epsilon_v) t_{f, v}}(2(1+\epsilon_v)\pi y)^2 y^{(1+\epsilon_v)s} \, \frac{dy}{y}. 
\end{equation*}
Recall that, for $a>0$ and $\Re(s)>0$, 
\begin{equation*}
\int_{0}^\infty K_{it}(2\pi ay)^2 y^{s}\, \frac{dy}{y}
= \frac{(\pi a)^{-s}}{8} \frac{\Gamma\left(\frac{s}{2}\right)\Gamma\left(\frac{s}{2}\right) \Gamma\left(\frac{s}{2}+it\right) \Gamma\left(\frac{s}{2}-it\right)}{\Gamma(s)}.
\end{equation*}
So we have 
\begin{multline}\label{e:Gsf_int}
\int_{F_\infty^\times} |W_{f_\infty}(a(y))|^2 \|y\|_\infty^{s-1} \, d_\infty^\times y
= \prod_{v\in S_\infty} \bigg\{2^{\epsilon_v} 
\frac{(\pi(1+\epsilon_v))^{-(1+\epsilon_v)s}}{8} 
\\ \times \frac{\Gamma\left(\frac{(1+\epsilon_v)s}{2}\right) \Gamma\left(\frac{(1+\epsilon_v)s}{2}\right) \Gamma\left(\frac{(1+\epsilon_v)s}{2}+i(1+\epsilon_v)t_{f, v}\right) \Gamma\left(\frac{(1+\epsilon_v)s}{2}-i(1+\epsilon_v)t_{f, v}\right)}
{\Gamma((1+\epsilon_v)s)}\bigg\}
\\ = G(s; f). 
\end{multline}

Taking $s=1$, 
\begin{equation*}
\rho(f) = \int_{F_\infty^\times} |W_{f_\infty}(a(y))|^2 \, d_\infty^\times y
= G(1; f)
= \prod_{v\in S_\infty}\rho(f_v), 
\end{equation*}
where 
\begin{multline*}
\rho(f_v)
= \prod_{v\in S_\infty} \bigg\{2^{\epsilon_v} 
\frac{(\pi(1+\epsilon_v))^{-(1+\epsilon_v)}}{8} 
\\ \times \frac{\Gamma\left(\frac{(1+\epsilon_v)}{2}\right) \Gamma\left(\frac{(1+\epsilon_v)}{2}\right) \Gamma\left(\frac{(1+\epsilon_v)}{2}+i(1+\epsilon_v)t_{f, v}\right) \Gamma\left(\frac{(1+\epsilon_v)}{2}-i(1+\epsilon_v)t_{f, v}\right)}
{\Gamma((1+\epsilon_v))}\bigg\}
\end{multline*}
When $F_v=\R$, 
\begin{equation*}
\rho(f_v) = \frac{1}{8}
\frac{\pi}{\cosh(\pi t_{f, v})}.
\end{equation*}
When $F_v=\C$, 
\begin{equation*}
\rho(f_v) = \frac{\Gamma(1+2it_{f, v})\Gamma(1-2it_{f, v})}{4(2\pi)^2} 
= \frac{t_{f_v}}{8\pi \sinh(2\pi t_{f, v})}.
\end{equation*}
So we get
\begin{equation*}
\rho(f) = \prod_{F_v=\R} \frac{\pi}{8\cosh(\pi t_{f, v})}
\prod_{F_v=\C} \frac{t_{f, v}}{8\pi \sinh(2\pi t_{f, v})}. 
\end{equation*}

Recall that the Eisenstein series $E(g; s)$ has a meromorphic continuation to $s\in \C$ and has only one simple pole at $s=1$ for $\Re(s)> \frac{1}{2}$ \cite[\S2.8]{g18}. 
Taking $s\to 1$, since we assume $\|f\|_2=1$, we get
\[1= \|f\|_2^2 = \frac{1}{c_E} \Res_{s=1} \bigg(\int_{\bar{\G}(F)\backslash \bar{\G}(\A)} |f(g)|^2 E(g; s)\, dV(g)\bigg) 
= \frac{1}{c_E} \rho(f) \Res_{s=1}\left(\sum_{\ba}\frac{|\rho_f(\ba)|^2}{N(\ba)^s}\right).\]
So we have 
\begin{equation*}
|\rho(f)|^{\frac{1}{2}} |\rho_f(\ba)| = O_\varepsilon(N(\ba)^{\theta+\varepsilon}), 
\end{equation*}
for any $\varepsilon>0$ \cite{MR2811610}.

\end{proof}
Recently, Silberman and Shem-Tov established Arithmetic Quantum Unique Ergodicity theorem for $\X_F$.
\begin{theorem}[\cite{zvi24}]\label{aque1}
Let $\phi$ be a $L^2$-normalized Maass--Hecke cuspform on $\X_F$. Then for any fixed smooth compactly supported function $u:\X_F \to \mathbb{C}$, we have
\[
\int_{\bar{\G}(F) \backslash \bar{\G}(\A)} |\phi(g)|^2 u(g) \, dV(g) = \frac{1}{\vol(\bar{\G}(F)\backslash \bar{\G}(\A))} 
\int_{\bar{\G}(F)\backslash \bar{\G}(\A)} u(g)\, dV(g)
+ o(1)
\]
as $t_f \to \infty$.
\end{theorem}
The proof of Theorem \ref{aque1} requires a nontrivial generalization of \cite{lin}, with ``no escape of mass''\cite{Z10} for Maass--Hecke cuspforms. The main difference between the case $F=\mathbb{Q}$ and the case $F\neq \mathbb{Q}$ is the existence of the $A$-invariant Hecke recurrent measures having positive entropy in the latter case. \cite{zvi24} uses infinitely many Hecke operators to rule out those measures being a component of a quantum limit. We conjecture here that it is possible to carry out the same result using only finitely many Hecke operators, hence generalizing \cite{MR3260861} to the number field cases.\footnote{We appreciate Lior Silberman for kindly explaining the work \cite{zvi24} to us.}
\begin{conjecture}\label{aque2}
There exists a finite set of primes $S$ such that the following is true. 
Let $f$ be an $L^2$-normalized Maass cuspform on $\X_F$ that is also a joint eigenfunction of the $\mathfrak{p}$th Hecke operator  for all $\mathfrak{p}\in S$. 
Then any weak-$\ast$-limit of $|f(g)|^2  \, dV(g) $ on $\bar{\G}(F) \backslash \bar{\G}(\A)$ is
\[
\frac{c}{\vol(\bar{\G}(F)\backslash \bar{\G}(\A))} \, dV(g)
\]
for some $c\in [0,1]$ as $t_f \to \infty$.
\end{conjecture}
We will use this conjecture to prove Theorem \ref{thm:QUE-multone}. To this end, following \cite[Proposition 3.7]{grs}, we first prove the following lemma.

\begin{lemma}\label{lem:grs}
Let $\varphi: (0, \infty)\to \R_{\geq 0}$ be a function which is supported on $[A, \infty)$. 
Assume that $\|f\|_2^2=1$. 
Fix $B>0$.
Assume that there exists $T_f>0$ such that $\frac{|t_{f, v}|}{T_f}\in [B^{-1}, B]$ for any $v\in S_{\infty}$. 
As $T_f\to \infty$, we get
\begin{multline}\label{e:fE=sum}
\int_{\bar{\G}(F) \backslash \bar{\G}(\A)} |f(g)|^2 E(g; \varphi) \, dV(g)
\\ = \sum_{\substack{\ba\\AN(\ba)\leq T_f^{d_F}}} |\rho_f(\ba)|^2 \int_{F_\infty^\times} |W_{f_\infty}(a(y))|^2 \|y\|_\infty^{-1} \varphi\left(\frac{\|y\|_\infty}{N(\ba)}\right)\, d_\infty^\times y 
+ o(1). 
\end{multline}
Assume that the weak-$\ast$-limit of $|f(g)|^2  \, dV(g) $ on $\bar{\G}(F) \backslash \bar{\G}(\A)$ as $T_f \to \infty$ is
\[\frac{c}{\vol(\bar{\G}(F)\backslash \bar{\G}(\A))} \, dV(g)\]
for some $c\in [0,1]$. 
Let $\alpha>0$ and take a constant $a>2(2\pi B)^{d_F}$.
Then, as $T_f\to \infty$, we get
\begin{equation}\label{e:sumrho_upperbd}
\frac{1}{B^{d_F} T_f^{d_F}}
N(\bd)^{-1} 
\sum_{\substack{\ba\\ \frac{N(\ba)}{N(\bd)} < \left(\frac{T_f}{2\pi B}\right)^{d_F}(a\alpha)^{-1}}} \rho(f) |\rho_f(\ba)|^2 
\leq  c\cdot \frac{c_{F}(a)}{(\alpha-\varepsilon)^{d_1+2d_2}}.
\end{equation}
Here 
\begin{equation*}
c_{F}(a) = \frac{\vol(F^\times\backslash \A^1)}{\vol(\bar{\G}(F)\backslash \bar{\G}(\A))} 
\frac{(d_1+d_2-1)!}{d_1+2d_2} \frac{(1-a^{-2})^{\frac{d_1}{2}}(1-a^{-1})^{\frac{d_2}{2}}}{(\log(a/2))^{d_1+d_2-1} \log(2)}
\frac{1}{\pi^{2d_2} (4\sqrt{\pi})^{d_1+d_2}}
\end{equation*}
where $d_1=\#S_{\R}$ and $d_2=\#S_{\C}$. 

\end{lemma}

\begin{proof}
Since $\varphi(H(g))$ is a function on $\BorelB(F)\N(\A) \CentZ(\A)\backslash \G(\A)$, we define 
\begin{equation*}
E(g; \varphi) = \sum_{\gamma\in \BorelB(F)\backslash \G(\A)} \varphi(H(\gamma g)). 
\end{equation*}
By \cite[Lemma (3.3)]{gj}, the series defining $E(g; \varphi)$ is a finite sum and $E(g; \varphi)$ is compactly supported function on $\bar{\G}(F)\backslash \bar{\G}(\A)$. 

We compute $\int_{\bar{G}(F) \backslash \bar{G}(\A)} |f_j(g)|^2 E(g; h)\, dV(g)$ in two different ways. 

First observe that, from the assumption, we have 
\begin{equation*}
\int_{\bar{\G}(F)\backslash \bar{\G}(\A)} |f(g)|^2 E(g; \varphi) \, dV(g)
= \frac{c}{\vol(\bar{\G}(F)\backslash \bar{\G}(\A))} 
\int_{\bar{\G}(F)\backslash \bar{\G}(\A)} E(g; \varphi)\, dV(g)
+ o(1)
\end{equation*}
as $\sum_{v\in S_\infty} |t_{f, v}|^2\to \infty$.
Note that $T_f\to \infty$ implies that $\sum_{v\in S_\infty} |t_{f, v}|^2\to \infty$. 
By unfolding the series $E(g; \varphi)$, we have 
\begin{align}
\int_{\bar{\G}(F)\backslash \bar{\G}(\A)} |f(g)|^2 E(g; \varphi) \, dV(g)
& = \frac{c}{\vol(\bar{\G}(F)\backslash \bar{\G}(\A))} 
\int_{\bar{\G}(F)\backslash \bar{\G}(\A)} E(g; \varphi)\, dV(g)
+ o(1) \nonumber 
\\ & = \frac{c}{\vol(\bar{\G}(F)\backslash \bar{\G}(\A))} 
\int_{\BorelB(F) \N(\A) \CentZ(\A) \backslash \G(\A)} \varphi(H(g))\, dV(g) + o(1) \nonumber 
\\ & = \frac{c}{\vol(\bar{\G}(F)\backslash \bar{\G}(\A))} 
\int_{F^\times\backslash \A^\times} \varphi(\|y\|) \|y\|^{-1}d^\times y + o(1) \label{e:f^2E_inner 1}
\end{align}
Here we use that $\vol(F\backslash\A)=1$. 

By the Mellin inverse transform, 
\begin{equation*}
\varphi(y) = \frac{1}{2\pi i}\int_{(\sigma)} \tilde{\varphi}(s) y^{-s}\, ds
\end{equation*}
and we have 
\begin{equation*}
E(g; \varphi)
= \frac{1}{2\pi i}\int_{(1+\varepsilon)} 
\tilde{\varphi}(-s) 
\sum_{\gamma\in \BorelB(F)\backslash \G(\A)} H(\gamma g)^s \, ds
= \frac{1}{2\pi i}\int_{(1+\varepsilon)} 
\tilde{\varphi}(-s) E(g, s)\, ds. 
\end{equation*}
By Lemma \ref{e:inner_f^2E}, 
\begin{multline}\label{e:f^2Einner_2}
\int_{\bar{\G}(F)\backslash \bar{\G}(\A)} |f(g)|^2 E(g; \varphi) \, dV(g)
= \frac{1}{2\pi i}\int_{(1+\varepsilon)} 
\tilde{\varphi}(-s) \int_{\bar{\G}(F)\backslash \bar{\G}(\A)} |f(g)|^2 E(g, s) \, dV(g)
\, ds
\\ = \sum_{\ba} |\rho_f(\ba)|^2 
\frac{1}{2\pi i}\int_{(1+\varepsilon)} 
\tilde{\varphi}(-s) N(\bd)^{s-1} N(\ba)^{-s} G(s; f) \, ds
\end{multline}

On the other hand, by \eqref{e:Gsf_int}, 
\[
G(s; f) = \int_{F_\infty^\times} |W_{f_\infty}(a(y))|^2 \|y\|_\infty^{s-1}\, d_\infty^\times y
\]
so by the convolution of Mellin transform, for each $\ba$, 
\begin{multline*}
\frac{1}{2\pi i} \int_{(1+\varepsilon)} \tilde{\varphi}(-s) N(\bd)^{s-1} N(\ba)^{-s} G(s; f)\, ds
\\ = \int_{F_\infty^\times} |W_{f_\infty}(a(y))|^2 
\|y\|_\infty^{-1} N(\bd)^{-1}
\frac{1}{2\pi i} \int_{(1+\varepsilon)} \tilde{\varphi}(-s) \left(\frac{\|y\|_\infty N(\bd)}{N(\ba)}\right)^{s}\, ds \, d_\infty^\times y
\\ = \int_{F_\infty^\times} |W_{f_\infty}(a(y))|^2 
\|y\|_\infty^{-1} N(\bd)^{-1} \varphi\left(\frac{\|y\|_\infty N(\bd)}{N(\ba)}\right) \, d_\infty^\times y.
\end{multline*}
Sicne $\varphi$ is supported on $[A, \infty)$, we have 
\begin{multline*} 
|\rho_f(\ba)|^2 
\int_{F_\infty^\times} |W_{f_\infty}(a(y))|^2 (N(\bd)\|y\|_\infty)^{-1} 
\varphi\left(\frac{\|y\|_\infty N(\bd)}{N(\ba)}\right)\, d_\infty^\times y
\\ \ll_\varphi \frac{1}{A} \frac{|\rho_f(\ba)|^2}{N(\ba)}
\int_{\|y\|_\infty \geq \frac{N(\ba)}{N(\bd)} A} |W_{f_\infty}(a(y))|^2
\, d_\infty^\times y.
\end{multline*}
If $\|y\|_\infty = \prod_{v\in S_\infty} |y_v|_v \geq \frac{N(\ba)}{N(\bd)} A$, then there exists at least one $v\in S_\infty$ such that $|y_v|\geq \left(\frac{N(\ba)}{N(\bd)}A\right)^{\frac{1}{d_F}}$.
Here $d_F = \sum_{v\in S_\infty}(1+\epsilon_v) = [F:\Q]$. 
So we have 
\begin{equation*}
\int_{\|y\|_\infty \geq N(\ba) A} |W_{f_\infty}(a(y))|^2
\, d_\infty^\times y
\ll \rho(f) \sum_{v\in S_\infty} 
\frac{1}{\rho(f_v)} 
\int_{|y_v|\geq \left(\frac{N(\ba)}{N(\bd)}A\right)^{\frac{1}{d_F}}} |W_{f_v}(a(y_v))|^2 \, d_v^\times y, 
\end{equation*}
where for each $v\in S_\infty$, by \eqref{e:rhof}, 
\begin{equation*}
\rho(f_v) = \int_{|y_v|_v>0} |W_{f_v}(a(y_v))|^2\, d_v^\times y
= \begin{cases}
\frac{\pi}{8\cosh(\pi t_{f, v})} & \text{ when } F_v=\R, \\
\frac{t_{f, v}}{8\pi \sinh(2\pi t_{f, v})} & \text{ when } F_v=\C 
\end{cases}
\end{equation*}
and $\rho(f) = \prod_{v\in S_\infty} \rho(f_v)$. 

For some $b>0$, we have 
\begin{equation*}
\int_{|y_v|\geq b} |W_{f_v}(a(y))|^2 \, d_v^\times y
= 2^{\epsilon_v} \int_{b}^\infty K_{i(1+\epsilon_v)t_{f, v}}(2\pi (1+\epsilon_v) y)^2 y^{1+\epsilon_v} \, \frac{dy}{y}.
\end{equation*}
By \cite[Corollary 3.2]{grs}, uniformly in $y>0$ and $r>0$ sufficiently large, one has
\begin{equation*}
K_{ir}(y) \ll \begin{cases}
(r^2-y^2)^{-\frac{1}{4}} e^{-\frac{\pi}{2}r} & \text{ when } 0<y<r-Cr^{\frac{1}{3}}, \\
y^{-\frac{1}{2}} e^{-y} & \text{ when } y>r+Cr^{\frac{1}{3}}, \\
r^{-\frac{1}{3}} e^{-\frac{\pi}{2}r} & \text{ when } |y-r|\leq Cr^{\frac{1}{3}}. 
\end{cases}
\end{equation*}
Here $C>0$ is a sufficiently large positive constant.
Assume that $b>|t_{f, v}|$.
Applying the second inequality, we get
\begin{multline*}
2^{\epsilon_v} \int_{b}^\infty K_{i(1+\epsilon_v)t_{f, v}}(2\pi (1+\epsilon_v) y)^2 y^{1+\epsilon_v} \, \frac{dy}{y}
\ll 2^{\epsilon_v} \int_b^\infty (2\pi(1+\epsilon_v)y)^{-1} 
e^{-2(2\pi(1+\epsilon_v)y} y^{1+\epsilon_v} \, \frac{dy}{y} 
\\ = 2^{\epsilon_v}(2\pi(1+\epsilon_v))^{-1} 
\int_b^\infty y^{-\epsilon_v} e^{-4\pi(1+\epsilon_v)y} \, \frac{dy}{y}
\\ \leq \frac{1}{4\pi} b^{-1-\epsilon_v} e^{-(4\pi(1+\epsilon_v)-c)b}
\int_{b}^\infty e^{-cy}\, dy
= \frac{1}{4\pi} b^{-1-\epsilon_v} e^{-(4\pi(1+\epsilon_v)-c)b}
\frac{1}{c}\int_{cb}^\infty e^{-y}\, dy
\\ = \frac{1}{4\pi} b^{-1-\epsilon_v} e^{-(4\pi(1+\epsilon_v))b}
\frac{1}{c}
\end{multline*}
for some fixed constant $c>0$. 
So for $b> |t_{f, v}|$, we have 
\begin{equation*}
\int_{|y_v|\geq b} |W_{f_v}(a(y))|^2 \, d_v^\times y
\ll \frac{1}{b^{1+\epsilon_v}} e^{-(4\pi(1+\epsilon_v))b}
\end{equation*}
Assume that $T_f$ is sufficiently large.
By taking $b=\left(\frac{N(\ba)}{N(\bd)} A\right)^{\frac{1}{d_F}}$, 
we get
\begin{multline*}
\frac{1}{A} \sum_{\substack{\ba\\ A\frac{N(\ba)}{N(\bd)}\geq T_f^{d_F}}} \frac{|\rho_f(\ba)|^2}{N(\ba)}
\int_{\|y\|_\infty \geq \frac{N(\ba)}{N(\bd)} A} |W_{f_\infty}(a(y))|^2
\, d_\infty^\times y
\\ \ll \frac{1}{A} \sum_{v\in S_\infty}
\sum_{\substack{\ba\\ A\frac{N(\ba)}{N(\bd)}\geq T_f^{d_F}}} \frac{\rho(f) |\rho_f(\ba)|^2}{N(\ba)} 
\frac{1}{\rho(f_v)}
\frac{e^{-4\pi(1+\epsilon_v)\left(A\frac{N(\ba)}{N(\bd)}\right)^{\frac{1}{d_F}}}}{\left(A\frac{N(\ba)}{N(\bd)}\right)^{\frac{(1+\epsilon_v)}{d_F}}}.
\end{multline*}
Since $\frac{1}{\rho(f_v)}\sim \frac{e^{(1+\epsilon_v)\pi |t_{f, v}|}}{|t_{f, v}|^{\epsilon_v}}$ (up to an explicit constant), 
\begin{equation*}
\ll \frac{1}{A} 
\sum_{\substack{\ba\\ AN(\ba)\geq t_f^{d_F}}} 
\frac{\rho(f)|\rho_f(\ba)|^2}{N(\ba)}
\sum_{v\in S_\infty} 
\frac{e^{-3\pi(1+\epsilon_v) \left(A\frac{N(\ba)}{N(\bd)}\right)^{\frac{1}{d_F}}}}{|t_{f, v}|^{\epsilon_v} \left(A\frac{N(\ba)}{N(\bd)}\right)^{\frac{(1+\epsilon_v)}{d_F}}}.
\end{equation*}
By \eqref{e:rhofrhof(ba)_size}
\begin{equation}\label{e:bounding_tail}
\ll \frac{1}{A} \sum_{\substack{\ba\\ A\frac{N(\ba)}{N(\bd)}\geq T_f^{d_F}}} N(\ba)^{-1+\theta+\varepsilon}
\sum_{v\in S_\infty} 
\frac{e^{-3\pi(1+\epsilon_v) \left(A\frac{N(\ba)}{N(\bd)}\right)^{\frac{1}{d_F}}}}{|t_{f, v}|^{\epsilon_v} \left(A\frac{N(\ba)}{N(\bd)}\right)^{\frac{(1+\epsilon_v)}{d_F}}}
= o(1)
\end{equation}
as $T_f \to \infty$. 

Applying to \eqref{e:f^2Einner_2} and combining with \eqref{e:f^2E_inner 1}, as $T_f\to \infty$, we get
\begin{multline}\label{e:series=QUE_lim_1}
N(\bd)^{-1} \sum_{\substack{\ba\\ A\frac{N(\ba)}{N(\bd)} \leq T_f^{d_F}}} 
|\rho_f(\ba)|^2 
\frac{1}{2\pi i}\int_{(1+\varepsilon)} 
\tilde{\varphi}(-s) \left(\frac{N(\ba)}{N(\bd)}\right)^{-s} G(s; f) \, ds
\\ = \int_{\bar{\G}(F)\backslash \bar{\G}(\A)} |f(g)|^2 E(g; \varphi) \, dV(g) + o(1)
\\ = \frac{c}{\vol(\bar{\G}(F)\backslash \bar{\G}(\A))} 
\int_{F^\times\backslash \A^\times} \varphi(\|y\|) \|y\|^{-1}d^\times y
+ o(1). 
\end{multline}

By applying Stirling's formula to \eqref{e:Gsf_def} (following \cite[(24)]{grs}), as $t_{f, v}\to \infty$, 
(and $|\Im(s)|$ is bounded) 
\begin{equation*}
\frac{\pi^{-s}}{8} 
\frac{\Gamma\left(\frac{s}{2}\right) \Gamma\left(\frac{s}{2}\right) \Gamma\left(\frac{s}{2}+it_{f, v}\right) \Gamma\left(\frac{s}{2}-it_{f, v}\right)}
{\Gamma(s)}
\sim 2\pi\frac{\pi^{-s}}{8} \frac{\Gamma\left(\frac{s}{2}\right)\Gamma\left(\frac{s}{2}\right)}{\Gamma(s)}
e^{-\pi|t_{f, v}|}|t_{f, v}|^{s-1}
\end{equation*}
when $F_v=\R$ and 
\begin{equation*}
\frac{(2\pi)^{-2s}}{4} 
\frac{\Gamma(s) \Gamma(s) \Gamma\left(s+i2t_{f, v}\right) \Gamma\left(s-i2t_{f, v}\right)}
{\Gamma(2s)}
\sim 2\pi\frac{(2\pi)^{-2s}}{4} \frac{\Gamma(s)\Gamma(s)}{\Gamma(2s)} e^{-2\pi|t_{f, v}|} |2t_{f, v}|^{2s-1}
\end{equation*}
when $F_v=\C$. 
By the duplication formula, 
\begin{equation*}
\frac{\Gamma\left(\frac{(1+\epsilon_v)s}{2}\right)^2}{\Gamma((1+\epsilon_v)s)}
= \sqrt{\pi} 2^{1-(1+\epsilon_v)s}
\frac{\Gamma\left(\frac{(1+\epsilon_v)s}{2}\right)}
{\Gamma\left(\frac{(1+\epsilon_v)s}{2}+\frac{1}{2}\right)}
\end{equation*}
As $T_f\to \infty$, we have 
\begin{multline*}
\frac{1}{2\pi i}\int_{(1+\varepsilon)} \tilde{\varphi}(-s) \left(\frac{N(\ba)}{N(\bd)}\right)^{-s} G(s; f)\, ds
\sim \frac{e^{-\pi \sum_{v\in s_{\infty}}(1+\epsilon_v)|t_{f, v}|}}{\prod_{v\in S_\infty}|t_{f, v}|}
(2^{-1}\pi^{\frac{3}{2}})^{d_1+d_2} 
\\ \times \frac{1}{2\pi i}\int_{(1+\varepsilon)} \tilde{\varphi}(-s)  
\bigg(\frac{\Gamma\left(\frac{s}{2}\right)}{\Gamma\left(\frac{s}{2}+\frac{1}{2}\right)}\bigg)^{d_1}
\bigg(\frac{\Gamma(s)}{\Gamma\left(s+\frac{1}{2}\right)}\bigg)^{d_2}
\bigg(\frac{(2\pi)^{d_1+2d_2} N(\ba)N(\bd)^{-1}}{\prod_{F_v=\R}|t_{f, v}|\prod_{F_v=\C} |t_{f, v}|^2}\bigg)^{-s} \, ds.
\end{multline*}
Here $d_1=\#S_{\R}$ and $d_2=\#S_{\C}$.
Recalling \eqref{e:Aphi_def}, we have 
\begin{equation*}
= \frac{e^{-\pi \sum_{v\in s_{\infty}}(1+\epsilon_v)|t_{f, v}|}}{\prod_{v\in S_\infty}|t_{f, v}|}
(2^{-1}\pi^{\frac{3}{2}})^{d_1+d_2} A_{\varphi}\bigg(\frac{(2\pi)^{d_1+2d_2} N(\ba)N(\bd)^{-1}}{\prod_{F_v=\R}|t_{f, v}|\prod_{F_v=\C} |t_{f, v}|^2}\bigg).
\end{equation*}
Then \eqref{e:series=QUE_lim_1} reads 
\begin{multline*}
N(\bd)^{-1} \sum_{\substack{\ba\\ A\frac{N(\ba)}{N(\bd)} \leq T_f^{d_F}}} |\rho_f(\ba)|^2 \frac{e^{-\pi \sum_{v\in S_{\infty}}(1+\epsilon_v)|t_{f, v}|}}{\prod_{v\in S_\infty}|t_{f, v}|}
(2^{-1}\pi^{\frac{3}{2}})^{d_1+d_2} A_{\varphi}\bigg(\frac{(2\pi)^{d_1+2d_2} N(\ba)N(\bd)^{-1}}{\prod_{F_v=\R}|t_{f, v}|\prod_{F_v=\C} |t_{f, v}|^2}\bigg)
\\ = \frac{c}{\vol(\bar{\G}(F)\backslash \bar{\G}(\A))}\int_{F^\times\backslash \A^\times} \varphi(\|y\|)\|y
\|^{-1}\, d^\times y + o(1).
\end{multline*}
Recalling \eqref{e:rhof}, as $t_{f, v}\to \infty$, 
\begin{equation*}
\rho(f) = \prod_{F_v=\R} \frac{\pi}{8\cosh(\pi t_{f, v})}
\prod_{F_v=\C} \frac{t_{f, v}}{8\pi \sinh(2\pi t_{f, v})}
\sim \prod_{v\in S_\infty} \frac{\pi^{1-2\epsilon_v} t_{f, v}^{\epsilon_v}}{8} e^{-\pi(1+\epsilon_v)t_{f, v}}, 
\end{equation*}
so we have 
\begin{multline*}
\frac{\pi^{2d_2} (4\sqrt{\pi})^{d_1+d_2}}{\prod_{F_v=\R}|t_{f, v}| \prod_{F_v=\C}|t_{f, v}|^2} 
N(\bd)^{-1} \sum_{\substack{\ba\\ A\frac{N(\ba)}{N(\bd)} \leq T_f^{d_F}}} \rho(f) |\rho_f(\ba)|^2 
A_{\varphi}\bigg(\frac{(2\pi)^{d_1+2d_2} N(\ba)N(\bd)^{-1}}{\prod_{F_v=\R}|t_{f, v}|\prod_{F_v=\C} |t_{f, v}|^2}\bigg)
\\ = \frac{c}{\vol(\bar{\G}(F)\backslash \bar{\G}(\A))}\int_{F^\times\backslash \A^\times} \varphi(\|y\|)\|y
\|^{-1}\, d^\times y
+ o(1).
\end{multline*}

Let 
\begin{equation*}
F_\infty^+ = \left\{y\in \A^\times:\, y_v=u\forall v\in S_\infty \text{ for } u>0 \text{ and } y_v=1 \forall v\in S_{\finite}\right\}
\end{equation*}
and
\begin{equation*}
\A^1 = \{y\in \A^\times:\, \|y\|=1\}.
\end{equation*}
Then we have
\begin{equation*}
F^\times \backslash \A^\times \cong F_\infty^+\times (F^\times\backslash \A^1). 
\end{equation*}
Let $d^1$ be a Haar measure on $F^\times \backslash \A^1$ such that $\int_{F^\times\backslash \A^\times}\, d^\times y = \int_{F^\times\backslash \A^1}\int_0^\infty \frac{du}{u} \, d^1y$.
We get
\begin{multline*}
\int_{F^\times\backslash \A^\times} \varphi(\|y\|)\|y\|^{-1}\, d^\times y
= \vol(F^\times \backslash \A^1) 
\int_0^\infty \varphi(u) u^{-(d_1+2d_2)} \, \frac{du}{u}
= \vol(F^\times \backslash \A^1) 
\tilde{\varphi}(-d_1-2d_2). 
\end{multline*}
Note that by definition, 
\begin{multline*}
\int_0^\infty A_\varphi(x) x^{d_1+2d_2} \frac{dx}{x}
= \widetilde{A_{\varphi}}(d_1+2d_2)
= \frac{\Gamma\left(\frac{d_1}{2}+d_2\right)^{d_1}}{\Gamma\left(\frac{d_1+1}{2}+d_2\right)^{d_1}}
\frac{\Gamma(d_1+2d_2)^{d_2}}{\Gamma\left(d_1+\frac{1}{2}+2d_2\right)^{d_2}}
\tilde\varphi(-d_1-2d_2).
\end{multline*}
Therefore we get
\begin{multline*}
\frac{c_{F, 1}}{\prod_{F_v=\R}|t_{f, v}| \prod_{F_v=\C}|t_{f, v}|^2}
N(\bd)^{-1} \sum_{\substack{\ba\\ A\frac{N(\ba)}{N(\bd)} \leq T_f^{d_F}}} \rho(f) |\rho_f(\ba)|^2 
A_{\varphi}\bigg(\frac{(2\pi)^{d_1+2d_2} N(\ba)N(\bd)^{-1}}{\prod_{F_v=\R}|t_{f, v}|\prod_{F_v=\C} |t_{f, v}|^2}\bigg)
\\ = \frac{c \vol(F^\times\backslash \A^1)}{\vol(\bar{\G}(F)\backslash \bar{\G}(\A))} 
\int_0^\infty A_\varphi(x) x^{d_1+2d_2}\, d^\times x + o(1). 
\end{multline*}
Here
\begin{equation*}
c_{F, 1} = \pi^{2d_2} (4\sqrt{\pi})^{d_1+d_2}
\frac{\Gamma\left(\frac{d_1}{2}+d_2\right)^{d_1}}
{\Gamma\left(\frac{d_1+1}{2}+d_2\right)^{d_1}}
\frac{\Gamma(d_1+2d_2)^{d_2}}{\Gamma\left(d_1+\frac{1}{2}+2d_2\right)^{d_2}}.
\end{equation*}

Now we assume that $0\leq \varphi(y) \leq 1$ on $(\alpha-\varepsilon, 2\alpha+\varepsilon)$, $\varphi(y)=1$ on $[\alpha, 2\alpha]$ and $\varphi(y)=0$ outside of $(\alpha-\varepsilon, 2\alpha+\varepsilon)$ for some small enough $\varepsilon>0$. 
Then $A\geq 2\alpha+\varepsilon$.
We have 
\begin{equation*}
\int_0^\infty \varphi(u) u^{-(d_1+2d_2)} \, d^\times 
\leq \int_{\alpha-\varepsilon}^{2\alpha+\varepsilon} u^{-(d_1+2d_2)} \, d^\times u 
\leq \int_{\alpha-\varepsilon}^\infty u^{-(d_1+2d_2)}\, d^\times u
= \frac{(\alpha-\varepsilon)^{-d_1-2d_2}}{d_1+2d_2}
\end{equation*}
and we get
\begin{multline*}
\frac{\pi^{2d_2} (4\sqrt{\pi})^{d_1+d_2}}{\prod_{F_v=\R}|t_{f, v}| \prod_{F_v=\C}|t_{f, v}|^2}
N(\bd)^{-1} \sum_{\substack{\ba\\ A\frac{N(\ba)}{N(\bd)} \leq T_f^{d_F}}} \rho(f) |\rho_f(\ba)|^2 
A_{\varphi}\bigg(\frac{(2\pi)^{d_1+2d_2} N(\ba)N(\bd)^{-1}}{\prod_{F_v=\R}|t_{f, v}|\prod_{F_v=\C} |t_{f, v}|^2}\bigg)
\\ \leq \frac{c\vol(F^\times\backslash \A^1)}{\vol(\bar{\G}(F)\backslash \bar{\G}(\A))} \frac{(\alpha-\varepsilon)^{-d_1-2d_2}}{d_1+2d_2}.
\end{multline*}
By Lemma \ref{lem:Aphi_lowerbd}, since $A_\varphi(x)\geq 0$, for $A\geq 2\alpha+\varepsilon$, we take $a>2(2\pi B)^{d_F}$ and get
\begin{multline*}
\frac{1}{(d_1+d_2-1)!} \frac{(\log(a/2))^{d_1+d_2-1} \log(2)}{(1-a^{-2})^{\frac{d_1}{2}}(1-a^{-1})^{\frac{d_2}{2}}}
\frac{\pi^{2d_2} (4\sqrt{\pi})^{d_1+d_2}}{\prod_{F_v=\R}|t_{f, v}| \prod_{F_v=\C}|t_{f, v}|^2}
\\ \times \frac{1}{N(\bd)} \sum_{\substack{\ba\\ \frac{N(\ba)}{N(\bd)} < \left(\frac{T_f}{2\pi B}\right)^{d_F} (a\alpha)^{-1}}} \rho(f) |\rho_f(\ba)|^2 
\\ \leq \frac{c\vol(F^\times\backslash \A^1)}{\vol(\bar{\G}(F)\backslash \bar{\G}(\A))} \frac{(\alpha-\varepsilon)^{-d_1-2d_2}}{d_1+2d_2}.
\end{multline*}

\end{proof}

\begin{lemma}\label{lem:Aphi_lowerbd}
Let $\alpha>0$ and $\varphi: (0, \infty) \to \R_{\geq 0}$ be a compactly supported smooth function. Assume that $\varphi(y)=1$ on $[\alpha, 2\alpha]$ and $\varphi(y)=0$ outside of $(\alpha-\varepsilon, 2\alpha+\varepsilon)$ for some small enough $\varepsilon>0$. 
Let 
\begin{equation*}
\tilde{\varphi}(s) = \int_0^\infty \varphi(y)y^{s}\, \frac{dy}{y}. 
\end{equation*}
For $d_1, d_2\in\Z_{\geq 1}$, define
\begin{equation}\label{e:Aphi_def}
A_\varphi(x) = \frac{1}{2\pi i} \int_{(1+\varepsilon)} \frac{\Gamma\left(\frac{s}{2}\right)^{d_1}}{\Gamma\left(\frac{s+1}{2}\right)^{d_1}}
\frac{\Gamma(s)^{d_2}}{\Gamma\left(s+\frac{1}{2}\right)^{d_2}}
\tilde{\varphi}(-s)x^{-s}\, ds.
\end{equation}
For $0<x<(a\alpha)^{-1}$ for some $a>1$, we have 
\begin{equation*}
A_\varphi(x) \geq \frac{1}{(d_1+d_2-1)!}
\frac{\left(\log(a/2)\right)^{d_1+d_2-1} \log(2)}{(1-a^{-2})^{\frac{d_1}{2}}(1-a^{-1})^{\frac{d_2}{2}}}.
\end{equation*}
\end{lemma}

\begin{proof}

Let 
\begin{equation*}
g_{n_1, n_2} (x) = \frac{1}{2\pi i} \int_{(1+\varepsilon) }
\frac{\Gamma\left(\frac{s}{2}\right)^{n_1}}{\Gamma\left(\frac{s+1}{2}\right)^{n_1}}
\frac{\Gamma(s)^{n_2}}{\Gamma\left(s+\frac{1}{2}\right)^{n_2}}
x^{-s}\, ds.
\end{equation*}
When $x>1$, we move the contour to $\Re(s)\to +\infty$ and get $g_{n_1, n_2}(x) = 0$.

First consider the case when $n_1=1$ and $n_2=0$. 
For $0<y<1$, we move the contour to $\Re(s)=-\infty$ and get (applying reflection formula)
\begin{equation*}
g_{1, 0}(x) = \sum_{\ell=0}^\infty \frac{(-1)^{\ell}}{\ell!}\frac{1}{\Gamma\left(\frac{1}{2}-\ell\right)}
x^{2\ell} 
= \frac{1}{\pi}
\sum_{\ell=0}^\infty \frac{\Gamma\left(\frac{1}{2}+\ell\right)}{\ell!} x^{2\ell}
= \frac{1}{\sqrt{\pi}}\frac{1}{\sqrt{1-x^2}}. 
\end{equation*}
So we have 
\begin{equation}\label{e:g1}
g_{1, 0} (x) = \begin{cases}
\frac{1}{\sqrt{\pi}}\frac{1}{\sqrt{1-x^2}} & \text{ if } x<1, \\
0 & \text{ if } x>1. 
\end{cases}
\end{equation}
When $n_1=0$ and $n_2=1$, by changing the variable $s$ to $\frac{s}{2}$, we get
\begin{equation*}
g_{0, 1}(x) = \frac{1}{2\pi i}\int_{(1+\varepsilon)} \frac{\Gamma(s)}{\Gamma\left(s+\frac{1}{2}\right)} x^{-s}\, ds
= \frac{1}{2} g_{1, 0}(\sqrt{x}). 
\end{equation*}
Let $g_{1, 0}(x) = g_1(x)$. 

We claim that, for any $n_1, n_2\geq 0$ with $n_1+n_2\geq 1$,
\begin{equation}\label{e:gn1n2_lbd}
g_{n_1, n_2}(x) \geq \frac{(\log(x^{-1}))^{n_1+n_2-1}}{(n_1+n_2-1)!} g_1(x)^{n_1} g_1(\sqrt{x})^{n_2}.
\end{equation}

We prove the claim by the induction hypothesis. 
When $n_1+n_2=1$ it is just an identity. 
Assume that $n_1\geq 1$ or $n_2\geq 1$. 
Let $m_1\in \{n_1-1, n_1\}$ and $m_2\in \{n_2-1, n_2\}$ and $m_1+m_2=n_1+n_2-1$. 
Let $\epsilon=0$ if $m_1=n_1-1$ and $\epsilon=1$ if $m_2=n_2-1$. 
Then 
\begin{multline*}
g_{n_1, n_2}(x) = \frac{1}{2\pi i} \int_{(1+\varepsilon)} \widetilde{g_{m_1, m_2}}(s) \widetilde{g_{1}}((1+\epsilon)s)\, ds
\\ = \int_0^\infty g_{m_1, m_2}(x/y) g_1(y^{\frac{1}{1+\epsilon}})\, \frac{dy}{y}
= \int_x^1 g_{m_1, m_2}(x/y) g_1(y^{\frac{1}{1+\epsilon}})\, \frac{dy}{y}
\\ \geq \frac{1}{(n_1+n_2-2)!} \int_x^1 (\log(y/x))^{n_1+n_2-2} g_1(x/y)^{m_1} g_1(\sqrt{x/y})^{m_2} g_1(y^{\frac{1}{1+\epsilon}})\, \frac{dy}{y}
\end{multline*}
By \eqref{e:g1}, for $0<x<y< 1$, we have 
\begin{equation*}
g_1((x/y)^{\frac{1}{1+\epsilon}}) 
= \frac{1}{\sqrt{\pi}} \frac{1}{\sqrt{1-(x/y)^{2}{1+\epsilon}}} 
> \frac{1}{\sqrt{\pi}}\frac{1}{\sqrt{1-x^{\frac{2}{1+\epsilon}}}} = g_1(x^{\frac{1}{1+\epsilon}}). 
\end{equation*}
We also have $g_1(y^{\frac{1}{1+\epsilon}}) \geq g_1(x^{\frac{1}{1+\epsilon}})$. 
So we get
\begin{equation*}
g_{n_1, n_2}(x) 
\geq \frac{1}{(n_1+n_2-2)!} g_{n_1}(x) g_{n_2}(\sqrt{x})
\int_x^1 (\log(y^{-1}))^{n_1+n_2-2} \, \frac{dy}{y}.
\end{equation*}
By changing the variable $\log(y^{-1})=t$, we have 
\begin{equation*}
\int_x^1 (\log(y^{-1}))^{n_1+n_2-2} \, \frac{dy}{y}
= \int_0^{\log(x^{-1})} t^{n_1+n_2-2} \, dt
= \frac{1}{n_1+n_2-1} (\log(x^{-1})^{n_1+n_2-1}.
\end{equation*}
Therefore 
\begin{equation*}
g_{n_1, n_2}(x) 
\geq \frac{1}{(n_1+n_2-1)!} g_{n_1}(x) g_{n_2}(\sqrt{x})
(\log(x^{-1})^{n_1+n_2-1}
\end{equation*}
as claimed. 

Now take $n_1=d_1$ and $n_2=d_2$. 
By the convolution for Mellin transform, and since both $g_{d_1, d_2}$ and $\varphi$ are non-negative, we have 
\begin{multline*}
A_{\varphi}(x) = \frac{1}{2\pi i}\int_{(1+\varepsilon)} \widetilde{g_{d_1, d_2}}(s) \tilde{\varphi}(-s) x^{-s}\, ds
= \int_0^\infty g_{d_1, d_2}(x/y) \varphi(1/y) \, \frac{dy}{y}
\\ = \int_0^\infty g_{d_1, d_2}(xy) \varphi(y)\, \frac{dy}{y}
\geq \int_{\alpha}^{2\alpha} g_{d_1, d_2}(xy) \, \frac{dy}{y}
= \int_{\alpha x}^{2\alpha x} g_{d_1, d_2}(y)\, \frac{dy}{y}.
\end{multline*} 
By \eqref{e:gn1n2_lbd},
\begin{equation*}
A_{\varphi}(x)
\geq \frac{1}{(d_1+d_2-1)!}
\int_{\alpha x}^{2\alpha x} 
g_{1}(y)^{d_1} g_1(\sqrt{y})^{d_2} (\log(y^{-1}))^{d_1+d_2-1} \, \frac{dy}{y}.
\end{equation*}
Since, for $\alpha x < y < 2\alpha x$, $g_1(y) \geq g_1(\alpha x)$ (and also $g_1(\sqrt{y})\geq g_1(\sqrt{\alpha x})$), we get
\begin{multline*}
A_\varphi(x) \geq \frac{g_1(\alpha x)^{d_1} g_1(\sqrt{\alpha x})^{d_2}}{(d_1+d_2-1)!} 
\int_{\alpha x}^{2\alpha x} (\log(y^{-1}))^{d_1+d_2-1} \, \frac{dy}{y}
\\ \geq \frac{g_1(\alpha x)^{d_1} g_1(\sqrt{\alpha x})^{d_2}}{(d_1+d_2-1)!} 
\left(\log((2\alpha x)^{-1})\right)^{d_1+d_2-1} \int_{\alpha x}^{2\alpha x} 1\, \frac{dy}{y}
\\ = \frac{g_1(\alpha x)^{d_1} g_1(\sqrt{\alpha x})^{d_2}}{(d_1+d_2-1)!} 
\left(\log((2\alpha x)^{-1})\right)^{d_1+d_2-1} \log(2).
\end{multline*}
Assuming that there exists $a>1$ such that $0< x< (a\alpha)^{-1}$, we get
\begin{equation*}
A_\varphi(x) \geq \frac{1}{(d_1+d_2-1)!}
\frac{\left(\log(a/2)\right)^{d_1+d_2-1} \log(2)}{(1-a^{-2})^{\frac{d_1}{2}}(1-a^{-1})^{\frac{d_2}{2}}}.
\end{equation*}
\end{proof}

\begin{proposition}\label{prop:escape}
Fix $\eta>0$ and $B>0$.  
Let $\{f_j\}_{j\geq 1}\subset L_0^2(\bar{\G}(F)\backslash \bar{\G}(\A)/\maxK)$ be a sequence of $L^2$-normalized Maass cusp forms such that $\frac{|t_{f_j, v}|}{t_j}\in [B, B^{-1}]$ for all $v\in S_\infty$ for some $t_j>0$ and $t_j\to +\infty$ as $j\to +\infty$. 

We furter assume that $f_j$ is an eigenfunction of Hecke operators for all $\bn$, $N(\bn)< \eta t_j^{d_F}$. Finally, assume Conjecture \ref{aque2}.
Then any weak-$\ast$-limit of $|f_j|^2\, dg$ as $j\to +\infty$ is of the form 
\begin{equation*}
\frac{c}{\vol(\bar{\G}(F) \backslash \bar{\G}(\A))} \, dg\quad\text{ for some constant }c\in (0, 1].
\end{equation*}
\end{proposition}

\begin{proof}
By Conjecture \ref{aque2}, any weak-$\ast$-limit of $|f_j|^2\, dg$ as $j\to +\infty$ is of the form 
\begin{equation*}
\frac{c}{\vol(\bar{\G}(F) \backslash \bar{\G}(\A))} \, dg\quad\text{ for some constant }c\in [0, 1].
\end{equation*}
Assume for contradiction that there is a sequence of such Maass forms $\{f_j\}_{j\geq 1}$ such that $|f_j|^2\, dg\to 0$ weakly as $j\to +\infty$. 

Let $\{\phi_j\}_{j\geq 1}\subset L^2_0(\bar{\G}(F)\backslash \bar{\G}(\A)/\maxK)$ be a sequence of all Maass--Hecke cusp forms. 
We further assume that $\|\phi_j\|_2^2=1$.
Then, by Theorem \ref{aque1},  
\begin{equation*}
|\phi_j|^2\, dg \to \frac{dg}{\vol(\bar{\G}(F)\backslash \bar{\G}(\A))}\text{ weakly as $j\to +\infty$.}
\end{equation*}

For each $j$, let $E_j\subset L_0^2(\bar{\G}(F)\backslash \bar{\G}(\A)/\maxK)$ be the subspace of all joint eigenfunctions of Laplace--Beltrami operators (for all $v\in S_\infty$) and Hecke operators for $\bn$, $N(\bn)<\eta t_j$ containing $f_j$. 
Then there exist at least one Maass--Hecke cusp form $\phi$ and $c_j$ such that $\phi\in E_j$ and 
\begin{equation*}
\rho_{f_j}(\bn) = c_j \rho_{\phi}(\bn)\quad\text{ for all } N(\bn) < \eta t_j^{d_F}.
\end{equation*}
This is because all Hecke operators and the Laplace--Beltrami operators are simultaneuously diagonalizable, i.e., $L_0^2(\bar{\G}(F)\backslash \bar{\G}(\A)/\maxK)$ is a direct sum of eigenspaces spanned by joint eigenfunctions and therefore so is $E_j$.

Let $\varphi: (0, \infty)\to \R_{\geq 0}$ be a function which is supported on $[\eta^{-1}, \infty)$. 
By \eqref{e:fE=sum} in Lemma \ref{lem:grs} as $t_{j}\to \infty$, we get
\begin{multline*}
\sum_{\substack{\ba\\ N(\ba)\leq \eta N(\bd) t_j^{d_F}}} |\rho_{\phi_j}(\ba)|^2 \int_{F_\infty^\times} 
|W_{\phi_{j, \infty}}(a(y))|^2 (N(\bd)\|y\|_\infty)^{-1} \varphi\left(\frac{\|y\|_\infty N(\bd)}{N(\ba)}\right) \, d_\infty^\times y 
\\ = \frac{1}{\vol(\bar{\G}(F)\backslash \bar{\G}(\A)} \int_{F^\times \backslash \A^\times} \varphi(\|y\|) \|y\|^{-1} \, d^\times y + o(1).
\end{multline*}
On the other hand, since $\rho_{f_j}(bn) = c_j \rho_{\phi_j}(\bn)$ for any $N(\bn) \leq \eta t_j^{d_F}$, we have 
\begin{multline*}
|c_j|^2 \sum_{\substack{\ba\\ N(\ba)\leq \eta N(\bd) t_j^{d_F}}} |\rho_{\phi_j}(\ba)|^2 \int_{F_\infty^\times} 
|W_{\phi_{j, \infty}}(a(y))|^2 (N(\bd)\|y\|_\infty)^{-1} \varphi\left(\frac{\|y\|_\infty N(\bd)}{N(\ba)}\right) \, d_\infty^\times y 
\\ = \sum_{\substack{\ba\\ N(\ba)\leq \eta N(\bd) t_j^{d_F}}} |\rho_{f_j}(\ba)|^2 \int_{F_\infty^\times} 
|W_{f_{j, \infty}}(a(y))|^2 (N(\bd)\|y\|_\infty)^{-1} \varphi\left(\frac{\|y\|_\infty N(\bd)}{N(\ba)}\right) \, d_\infty^\times y 
\\ = \frac{c}{\vol(\bar{G}(F)\backslash \bar{\G}(\A)} \int_{F^\times \backslash \A^\times} \varphi(\|y\|) \|y\|^{-1} \, d^\times y + o(1)
\end{multline*}
and by assumption, $|c|=o(1)$.
Comparing the above, as $t_j\to +\infty$, we get
\begin{equation*}
|c_j|^2 
\frac{1}{\vol(\bar{G}(F)\backslash \bar{\G}(\A)} \int_{F^\times \backslash \A^\times} \varphi(\|y\|) \|y\|^{-1} \, d^\times y
= o(1),
\end{equation*}
and so
\begin{equation*}
|c_j| = o(1) \quad \text{ as } t_j\to +\infty. 
\end{equation*}
Since both $f_j$ and $\phi_j$ are $L^2$-normalized, we get
\begin{align*}
1 & = \|f_j\|_2^2
= \|f_j -c_j \phi_j + c_j\phi_j\|_2^2
\\ & \leq 2\|f_j-c_j\phi_j\|_2^2 + 2|c_j|^2\|\phi_j\|_2^2
= 2\|f_j-c_j\phi_j\|_2^2 + 2|c_j|^2
= 2\|f_j-c_j\phi_j\|_2^2+ o(1)
\end{align*}
as $t_j\to + \infty$, since $|c_j|=o(1)$. 

Take a Siegel $\Siegel=\{n(x)a(y)\kappa\in \bar{\G}(\A): x\in F\backslash \A, \, \|y\|\geq T_F, \, \kappa\in \K\}$. 
By \cite[Corollary 2.2.8]{g18}, there exists $T_F>0$ such that $\bar{\G}(F)\Siegel = \bar{\G}(\A)$.
Following the proof of Lemma \ref{lem:grs}, we have 
\begin{multline*}
\|f_j-c_j\phi_j\|_2^2 
\leq \int_{F^\times\backslash \A^\times}\int_{F\backslash \A} 
1_{\geq T_F}(\|y\|) 
|f_j(n(x)a(y))- c_j\phi(n(x)a(y))|^2 \|y\|^{-1}
\, dx\, d^\times y
\\ = \sum_{\ba} |\rho_{f_j}(\ba)-c_j \rho_{\phi_j}(\ba)|^2 
\int_{F_\infty^\times} 1_{\geq T_F}\left(\frac{\|y\|_\infty N(\bd)}{N(\ba)}\right) |W_{\phi_{j, \infty}}(a(y))|^2 (N(\bd) \|y\|_\infty)^{-1}\, d_\infty^\times y
\\ = \sum_{\substack{\ba\\ N(\ba) > \eta t_j^{d_F}}} |\rho_{f_j}(\ba)-c_j \rho_{\phi_j}(\ba)|^2 
\int_{F_\infty^\times} 1_{\geq T_F}\left(\frac{\|y\|_\infty N(\bd) }{N(\ba)}\right) |W_{\phi_{j, \infty}}(a(y))|^2 (N(\bd) \|y\|_\infty)^{-1}\, d_\infty^\times y.
\end{multline*}
Here $1_{\geq T_f}(y)=1$ when $y\geq T_F$ and $0$ otherwise. 

Following the proof of \eqref{e:fE=sum} in Lemma \ref{lem:grs}, 
\begin{multline*}
\sum_{\substack{\ba\\ N(\ba) > \eta t_j^{d_F}}} |\rho_{f_j}(\ba)-c_j \rho_{\phi_j}(\ba)|^2 
\int_{F_\infty^\times} 1_{\geq T_F}\left(\frac{\|y\|_\infty N(\bd)}{N(\ba)}\right) |W_{\phi_{j, \infty}}(a(y))|^2 (N(\bd)\|y\|_\infty)^{-1}\, d_\infty^\times y
\\ = \sum_{\substack{\ba\\ \eta t_j^{d_F} < N(\ba) \leq N(\bd) \frac{t_j^{d_F}}{T_F}}} |\rho_{f_j}(\ba)-c_j \rho_{\phi_j}(\ba)|^2 
\int_{F_\infty^\times} 1_{\geq T_F}\left(\frac{\|y\|_\infty N(\bd)}{N(\ba)}\right) |W_{\phi_{j, \infty}}(a(y))|^2 (N(\bd) \|y\|_\infty)^{-1}\, d_\infty^\times y 
\\ + o(1). 
\end{multline*}
Since the integrand is positive, 
\begin{multline*}
\int_{F_\infty^\times} 1_{\geq T_F}\left(\frac{\|y\|_\infty N(\bd)}{N(\ba)}\right) |W_{\phi_{j, \infty}}(a(y))|^2 (N(\bd)\|y\|_\infty)^{-1}\, d_\infty^\times y 
\\ \leq \frac{1}{N(\ba) T_F}
\int_{F_\infty^\times} |W_{\phi_{j, \infty}}(a(y))|^2\, d_\infty^\times y
= \frac{1}{N(\ba) T_F} \rho(\phi_j).
\end{multline*}
Since $\eta t_j^{d_F} T_F < T_F N(\ba) \leq N(\bd) t_j^{d_F}$, we get
\begin{multline*}
\|f_j-c_j\phi_j\|_2^2
\leq \frac{1}{\eta t_j^{d_F} T_F} \sum_{\substack{\ba\\ N(\ba) \leq N(\bd) \frac{t_j^{d_F}}{T_F}}}
\rho(\phi_j) |\rho_{f_j}(\ba) - c_j\rho_{\phi_j}(\ba)|^2 + o(1)
\\ \leq \frac{2}{T_F\eta t_j^{d_F}} \sum_{\substack{\ba\\ N(\ba) \leq N(\bd) \frac{t_j^{d_F}}{T_F}}} \rho(f_j)|\rho_{f_j}(\ba)|^2
+ \frac{2}{T_F\eta t_j^{d_F}} |c_j|^2 \sum_{\substack{\ba\\ N(\ba) \leq N(\bd) \frac{t_j^{d_F}}{T_F}}} \rho(\phi_j)|\rho_{\phi_j}(\ba)|^2
+ o(1).
\end{multline*}
Here we use $\rho(\phi_j) = \rho(f_j)$.
By \eqref{e:sumrho_upperbd} in Lemma \ref{lem:grs}, 
taking $\alpha>0$ such that $T_F= (2\pi B)^{d_F} \alpha a$ for some $a>2(2\pi B)^{d_F}$, we get 
\begin{equation*}
\leq  \frac{B^{d_F}}{T_F\eta}
\frac{c_F(a)}{(\alpha-\varepsilon)^{d_F}}\big(c + |c_j|^2 c_0\big) + o(1)
= o(1)
\end{equation*}
as $t_j\to +\infty$. 
Therefore, 
\begin{equation*}
1\leq 2\|f_j-c_j\phi_j\|_2^2 + o(1)
\leq o(1) \quad\text{ as } t_j\to +\infty, 
\end{equation*}
which is a contradiction. 
\end{proof}

Now we are ready to prove Theorem \ref{thm:QUE-multone}.

\begin{proof}[Proof of Theorem \ref{thm:QUE-multone}]
Assume a contradiction. 
For any sufficiently large $\tilde{t}>0$, there exist Maass--Hecke cuspforms $\phi_1, \phi_2\in L^2_0(\bar{\G}(F)\backslash \bar{\G}(\A)/\maxK)$ with the same Laplace--Beltrami eigenvalues with parameter $t_v$, such that they are not the same up to constant multiple and $t>\tilde{t}$ such that 
\begin{itemize}
\item $\frac{|t_v|}{t}\in [B^{-1}, B]$ for any $v\in S_\infty$; 
\item $\phi_1$ and $\phi_2$  share the same eigenvalues  for $N(\bn) \leq \eta t^{d_F}$. 
\end{itemize}
This assumption implies that there exist infinitely many pairs of Maass cusp forms $(\phi_{1, j}, \phi_{2, j})$ (for $j\geq 1$) and $t_j>0$ that satisfy the above conditions. 

In particular, we can choose a sequence of such pairs such that $t_j\to +\infty$ as $j\to +\infty$. 
For each $j$, let $W_{j, \infty}(a(y)) := W_{\phi_{1, j}, \infty}(a(y))= W_{\phi_{2, j}, \infty}(a(y))$ for $y\in F_\infty^{\times}$. 

Since $\phi_{1, j}$ and $\phi_{2, j}$ share the same Hecke eigenvalues for $N(\bn)\leq \eta t_j^{d_F}$, there exists $c_j$ such that $\rho_{\phi_{1, j}}(\bn)=c_j\rho_{\phi_{2, j}}(\bn)$ for $N(\bn)\leq \eta t_j^{d_F}$.

By our assumption, $\phi_{1, j}-c_j \phi_{2, j}\neq 0$ for all $j$. 
Let 
\begin{equation*}
f_j := \frac{\phi_{1, j}-c_j \phi_{2, j}}{\|\phi_{1, j}-c_j \phi_{2, j}\|_2}. 
\end{equation*}
Then $f_j$ is a Maass form, $L^2$-normalized and a joint Hecke eigenfunctions for $N(\bn)\leq \eta t_j^{d_F}$.
So we have constructed a sequence $\{f_j\}_{j\geq 1}$ that satisfies the hypothesis of Proposition \ref{prop:escape}.
We have $W_{\psi_j, \infty}(a(y)) = W_{j, \infty}(a(y))$ and 
\begin{equation*}
\rho_{f_j}(\ba) = \frac{\rho_{\phi_{1, j}}(\ba)-c_j \rho_{\phi_{2, j}}(\ba)}{\|\phi_{1, j}-c_j\phi_{2, j}\|_2}
\quad \text{ for all }\ba.
\end{equation*}
Then $\rho_{f_j}(\bn)=0$ for $N(\bn)\leq \eta t_j^{d_F}$. 

Take $T_F>0$ as in the proof of Proposition \ref{prop:escape}.
Whenever $A>T_F$, 
\begin{multline*}
\int_{\bar{\G}(F)\backslash \bar{\G}(\A)} 1_{\geq A}(g) |f_j(g)|^2 \, dg
= \int_{F^\times \backslash \A^\times} 1_{\geq A}(\|y\|)
\int_{F\backslash \A} |f_j(n(x)a(y))|^2 \, dx\, \|y\|^{-1}\, d^\times y
\\ = \sum_{\ba} |\rho_{f_j}(\ba)|^2 \int_{F_\infty^\times} 1_{\geq A}\left(\frac{\|y\|_\infty N(\bd)}{N(\ba)}\right) |W_{j, \infty}(a(y))|^2 (N(\bd)\|y\|_\infty)^{-1} \, d_\infty^\times y
\\ \leq \frac{1}{A} \sum_{\substack{\ba\\ N(\ba) > \eta t_j^{d_F}}} \frac{|\rho_{f_j}(\ba)|^2 }{N(\ba)}
\int_{F_\infty^\times} 1_{\geq A}\left(\frac{\|y\|_\infty N(\bd)}{N(\ba)}\right) |W_{j, \infty}(a(y))|^2 \, d_\infty^\times y.
\end{multline*}
Here we use that $\rho_{f_j}(\ba)=0$ unless $N(\ba) > \eta t_j^{d_F}$.
Following the argument in the proof of Lemma \ref{lem:grs}, p.\pageref{e:bounding_tail}, 
by taking $A\geq \eta^{-1}$, 
\begin{equation*}
\frac{1}{A } \sum_{\substack{\ba\\ N(\ba)> \eta t_j^{d_F}}} \frac{|\rho_{f_j}(\ba)|^2 }{N(\ba)}
\int_{F_\infty^\times} 1_A\left(\frac{\|y\|_\infty N(\bd)}{N(\ba)}\right) |W_{j, \infty}(a(y))|^2 \, d_\infty^\times y = o(1), 
\end{equation*}
as $t_j\to +\infty$. 
Therefore we conclude that, for any sufficiently large $A\geq \eta^{-1}$,  
\begin{equation*}
\int_{\bar{\G}(F)\backslash \bar{\G}(\A)} 1_A(g) |f_j(g)|^2 \,dg = o(1)\quad\text{ as } t_j\to +\infty. 
\end{equation*}
This implies that, for any compact set $U\subset \{g\in \bar{\G}(F)\backslash \bar{\G}(\A):\, H(g)\geq \eta^{-1}\}$, 
\begin{equation*}
\int_{U} |f_j(g)|^2 \,dg=o(1) \quad\text{ as } t_j\to +\infty. 
\end{equation*}
By Proposition \ref{prop:escape}, such a sequence cannot exist, so it is a contradiction. 
\end{proof}

\bibliographystyle{alpha}
\bibliography{bibfile}

\begin{thebibliography}{CLLL24}

\bibitem[B\'77]{be77}
Pierre~H. B\'{e}rard.
\newblock On the wave equation on a compact {R}iemannian manifold without
  conjugate points.
\newblock {\em Math. Z.}, 155(3):249--276, 1977.

\bibitem[BB11]{MR2811610}
Valentin Blomer and Farrell Brumley.
\newblock On the {R}amanujan conjecture over number fields.
\newblock {\em Ann. of Math. (2)}, 174(1):581--605, 2011.

\bibitem[Ber77]{MR489542}
M.~V. Berry.
\newblock Regular and irregular semiclassical wavefunctions.
\newblock {\em J. Phys. A}, 10(12):2083--2091, 1977.

\bibitem[BK11]{BK11}
Andrew~R. Booker and M.~Krishnamurthy.
\newblock A strengthening of the {${\rm GL}(2)$} converse theorem.
\newblock {\em Compos. Math.}, 147(3):669--715, 2011.

\bibitem[BL14]{MR3260861}
Shimon Brooks and Elon Lindenstrauss.
\newblock Joint quasimodes, positive entropy, and quantum unique ergodicity.
\newblock {\em Invent. Math.}, 198(1):219--259, 2014.

\bibitem[Bru06]{Bru06}
Farrell Brumley.
\newblock Effective multiplicity one on {${\rm GL}_N$} and narrow zero-free
  regions for {R}ankin-{S}elberg {$L$}-functions.
\newblock {\em Amer. J. Math.}, 128(6):1455--1474, 2006.

\bibitem[Bum97]{Bump97}
Daniel Bump.
\newblock {\em Automorphic forms and representations}, volume~55 of {\em
  Cambridge Studies in Advanced Mathematics}.
\newblock Cambridge University Press, Cambridge, 1997.

\bibitem[CLLL24]{CLLL}
Dohoon Choi, Min Lee, Youngmin Lee, and Subong Lim.
\newblock The number of automorphic representations of {$GL_2$} with
  exceptional eigenvalues.
\newblock {\em arXiv:2402.11761 [math.NT]}, 2024.

\bibitem[DK00]{DK}
W.~Duke and E.~Kowalski.
\newblock A problem of {L}innik for elliptic curves and mean-value estimates
  for automorphic representations.
\newblock {\em Invent. Math.}, 139(1):1--39, 2000.
\newblock With an appendix by Dinakar Ramakrishnan.

\bibitem[Gar18]{g18}
Paul Garrett.
\newblock {\em Modern analysis of automorphic forms by example. {V}ol. 1},
  volume 173 of {\em Cambridge Studies in Advanced Mathematics}.
\newblock Cambridge University Press, Cambridge, 2018.

\bibitem[Gel75]{MR379375}
Stephen~S. Gelbart.
\newblock {\em Automorphic forms on ad\`ele groups}, volume No. 83 of {\em
  Annals of Mathematics Studies}.
\newblock Princeton University Press, Princeton, NJ; University of Tokyo Press,
  Tokyo, 1975.

\bibitem[GJ79]{gj}
Stephen Gelbart and Herv\'{e} Jacquet.
\newblock Forms of {${\rm GL}(2)$} from the analytic point of view.
\newblock In {\em Automorphic forms, representations and {$L$}-functions
  ({P}roc. {S}ympos. {P}ure {M}ath., {O}regon {S}tate {U}niv., {C}orvallis,
  {O}re., 1977), {P}art 1}, Proc. Sympos. Pure Math., XXXIII, pages 213--251.
  Amer. Math. Soc., Providence, R.I., 1979.

\bibitem[GK75]{gk}
I.~M. Gelfand and D.~A. Kajdan.
\newblock Representations of the group {${\rm GL}(n,K)$} where {$K$} is a local
  field.
\newblock In {\em Lie groups and their representations ({P}roc. {S}ummer
  {S}chool, {B}olyai {J}\'{a}nos {M}ath. {S}oc., {B}udapest, 1971)}, pages
  95--118. Halsted Press, New York-Toronto, Ont., 1975.

\bibitem[GRS13]{grs}
Amit Ghosh, Andre Reznikov, and Peter Sarnak.
\newblock Nodal domains of {M}aass forms {I}.
\newblock {\em Geom. Funct. Anal.}, 23(5):1515--1568, 2013.

\bibitem[HM97]{MR1610859}
Haruzo Hida and Yoshitaka Maeda.
\newblock Non-abelian base change for totally real fields.
\newblock {\em Pacific J. Math.}, pages 189--217, 1997.
\newblock Olga Taussky-Todd: in memoriam.

\bibitem[H{\"o}r68]{Lars}
Lars H{\"o}rmander.
\newblock The spectral function of an elliptic operator.
\newblock {\em Acta Math.}, 121:193--218, 1968.

\bibitem[Hun91]{hunt}
Jonathan Huntley.
\newblock Spectral multiplicity on products of hyperbolic spaces.
\newblock {\em Proc. Amer. Math. Soc.}, 111(1):1--12, 1991.

\bibitem[IS00]{is}
H.~Iwaniec and P.~Sarnak.
\newblock Perspectives on the analytic theory of {$L$}-functions.
\newblock {\em Geom. Funct. Anal.}, pages 705--741, 2000.
\newblock GAFA 2000 (Tel Aviv, 1999).

\bibitem[Jac72]{Jac72}
Herv\'{e} Jacquet.
\newblock {\em Automorphic forms on {${\rm GL}(2)$}. {P}art {II}}.
\newblock Lecture Notes in Mathematics, Vol. 278. Springer-Verlag, Berlin-New
  York, 1972.

\bibitem[JPSS83]{jpss}
H.~Jacquet, I.~I. Piatetskii-Shapiro, and J.~A. Shalika.
\newblock Rankin-{S}elberg convolutions.
\newblock {\em Amer. J. Math.}, 105(2):367--464, 1983.

\bibitem[JS81a]{MR623137}
H.~Jacquet and J.~A. Shalika.
\newblock On {E}uler products and the classification of automorphic forms.
  {II}.
\newblock {\em Amer. J. Math.}, 103(4):777--815, 1981.

\bibitem[JS81b]{MR618323}
H.~Jacquet and J.~A. Shalika.
\newblock On {E}uler products and the classification of automorphic
  representations. {I}.
\newblock {\em Amer. J. Math.}, 103(3):499--558, 1981.

\bibitem[Kim03]{k03}
Henry~H. Kim.
\newblock Functoriality for the exterior square of {${\rm GL}_4$} and the
  symmetric fourth of {${\rm GL}_2$}.
\newblock {\em J. Amer. Math. Soc.}, 16(1):139--183, 2003.
\newblock With appendix 1 by Dinakar Ramakrishnan and appendix 2 by Kim and
  Peter Sarnak.

\bibitem[Li10]{XiaLi10}
Xiannan Li.
\newblock Upper bounds on {$L$}-functions at the edge of the critical strip.
\newblock {\em Int. Math. Res. Not. IMRN}, (4):727--755, 2010.

\bibitem[Lin06]{lin}
Elon Lindenstrauss.
\newblock Invariant measures and arithmetic quantum unique ergodicity.
\newblock {\em Ann. of Math. (2)}, 163(1):165--219, 2006.

\bibitem[LRS99]{lrs}
Wenzhi Luo, Ze\'{e}v Rudnick, and Peter Sarnak.
\newblock On the generalized {R}amanujan conjecture for {${\rm GL}(n)$}.
\newblock In {\em Automorphic forms, automorphic representations, and
  arithmetic ({F}ort {W}orth, {TX}, 1996)}, volume~66 of {\em Proc. Sympos.
  Pure Math.}, pages 301--310. Amer. Math. Soc., Providence, RI, 1999.

\bibitem[Mor85]{MR778093}
Carlos~J. Moreno.
\newblock Analytic proof of the strong multiplicity one theorem.
\newblock {\em Amer. J. Math.}, 107(1):163--206, 1985.

\bibitem[Pal12]{zbMATH06321163}
Marc~R. Palm.
\newblock {\em Explicit {{\(\mathrm{GL}(2)\)}} trace formulas and uniform,
  mixed {Weyl} laws}.
\newblock G{\"o}ttingen: Univ. G{\"o}ttingen (Diss.), 2012.

\bibitem[PS79]{MR546599}
I.~I. Piatetski-Shapiro.
\newblock Multiplicity one theorems.
\newblock In {\em Automorphic forms, representations and {$L$}-functions
  ({P}roc. {S}ympos. {P}ure {M}ath., {O}regon {S}tate {U}niv., {C}orvallis,
  {O}re., 1977), {P}art 1}, volume XXXIII of {\em Proc. Sympos. Pure Math.},
  pages 209--212. Amer. Math. Soc., Providence, RI, 1979.

\bibitem[Ram00]{R00}
Dinakar Ramakrishnan.
\newblock Modularity of the {R}ankin-{S}elberg {$L$}-series, and multiplicity
  one for {${\rm SL}(2)$}.
\newblock {\em Ann. of Math. (2)}, 152(1):45--111, 2000.

\bibitem[Sar02]{sr}
Peter Sarnak.
\newblock Letter to {Z}. {R}udnick on multiplicities of eigenvalues for the
  modular surface.
\newblock {\em publications.ias.edu/sarnak/section/515}, 2002.

\bibitem[Sel56]{MR0088511}
A.~Selberg.
\newblock Harmonic analysis and discontinuous groups in weakly symmetric
  {R}iemannian spaces with applications to {D}irichlet series.
\newblock {\em J. Indian Math. Soc. (N.S.)}, 20:47--87, 1956.

\bibitem[Sou10]{sou}
Kannan Soundararajan.
\newblock Quantum unique ergodicity for {${\rm SL}_2(\mathbb
  Z)\backslash\mathbb H$}.
\newblock {\em Ann. of Math. (2)}, 172(2):1529--1538, 2010.

\bibitem[STS24]{zvi24}
Zvi Shem-Tov and Lior Silberman.
\newblock Arithmetic quantum unique ergodicity for products of hyperbolic 2-
  and 3-spaces.
\newblock {\em To appear in Journal d'Analyse Mathematique}, 2024.

\bibitem[You23]{MR4635352}
Matthew~P. Young.
\newblock On the spectral large sieve inequality for symmetric-squares.
\newblock {\em Forum Math.}, 35(5):1221--1236, 2023.

\bibitem[Zam12]{Z10}
Asif~Ali Zaman.
\newblock {\em Escape of {M}ass on {H}ilbert modular varieties}.
\newblock The University of British Columbia, Vancouver, 2012.
\newblock MSc dissertation.

\end{thebibliography}
\end{document}